\documentclass[a4paper,reqno, 10pt]{amsart}

\usepackage{amssymb}
\usepackage{amstext}
\usepackage{amsmath}
\usepackage{amscd}
\usepackage{amsthm}
\usepackage{amsfonts}
\usepackage{dsfont}
\usepackage{enumerate}
\usepackage{graphicx}
\usepackage{latexsym}
\usepackage{mathrsfs}
\usepackage[all]{xy}
\xyoption{all}

\usepackage{pstricks}
\usepackage{lscape}

\newtheorem{theorem2}{Theorem}

\newtheorem{theorem}{Theorem}[section]

\newtheorem{corollary}[theorem]{Corollary}
\newtheorem{lemma}[theorem]{Lemma}
\newtheorem{proposition}[theorem]{Proposition}
\newtheorem{definition-proposition}[theorem]{Definition-Proposition}
\newtheorem{question}[theorem]{Question}
\newtheorem{conjecture}[theorem]{Conjecture}

\theoremstyle{definition}
\newtheorem{definition}[theorem]{Definition}
\newtheorem{example}[theorem]{Example}

\theoremstyle{remark}
\newtheorem{remark}[theorem]{Remark}

\newcommand{\CH}{\operatorname{CH}\nolimits}

\renewcommand{\hom}{\mathrm{hom}}
\newcommand{\alg}{\mathrm{alg}}
\newcommand{\C}{\mathbf{C}}

\newcommand{\Q}{\mathbf{Q}}
\newcommand{\Z}{\mathbf{Z}}

\newcommand{\N}{\widetilde{N}}

\newcommand{\im}{\mathrm{Im}\,}
\newcommand{\End}{\mathrm{End}}
\newcommand{\Hom}{\mathrm{Hom}}
\newcommand{\Gal}{\mathrm{Gal}}
\newcommand{\id}{\mathrm{id}}

\renewcommand\r{\rightarrow}

\renewcommand{\ker}{\mathrm{Ker}\,}

\newcommand\h{\mathfrak{h}}

\newcommand{\arr}[1]{\overset{#1}{\longrightarrow}}

\begin{document}

\title{Remarks on motives of abelian type}
\author{Charles Vial}
\date{}

\address{DPMMS, University of Cambridge, Wilberforce Road, Cambridge,
  CB3 0WB, UK} \email{c.vial@dpmms.cam.ac.uk}
\urladdr{http://www.dpmms.cam.ac.uk/~cv248/}
\thanks{ The author is
  supported by an EPSRC Early Career Fellowship EP/K005545/1.}

\begin{abstract}
  A motive over a field $k$ is of abelian type if it belongs to the
  thick and rigid subcategory of Chow motives spanned by the motives
  of abelian varieties over $k$. This paper contains three sections of
  independent interest. First, we show that a motive which becomes of
  abelian type after a base field extension of algebraically closed
  fields is of abelian type. Given a field extension $K/k$ and a
  motive $M$ over $k$, we also show that $M$ is finite-dimensional if
  and only if $M_K$ is finite-dimensional. As a corollary, we obtain
  Chow--K\"unneth decompositions for varieties that become isomorphic
  to an abelian variety after some field extension.  Second, let
  $\Omega$ be a universal domain containing $k$. We show that Murre's
  conjectures for motives of abelian type over $k$ reduce to Murre's
  conjecture (D) for products of curves over $\Omega$. In particular,
  we show that Murre's conjecture (D) for products of curves over
  $\Omega$ implies Beauville's vanishing conjecture on abelian
  varieties over $k$. Finally, we give criteria on Chow groups for a
  motive to be of abelian type. For instance, we show that $M$ is of
  abelian type if and only if the total Chow group of algebraically
  trivial cycles $\CH_*(M_\Omega)_\alg$ is spanned, via the action of
  correspondences, by the Chow groups of products of curves. We also
  show that a morphism of motives $f: N \r M$, with $N$
  finite-dimensional, which induces a surjection $f_* :
  \CH_*(N_\Omega)_\alg \r \CH_*(M_\Omega)_\alg$ also induces a
  surjection $f_* : \CH_*(N_\Omega)_\hom \r \CH_*(M_\Omega)_\hom$ on
  homologically trivial cycles.
\end{abstract}
\maketitle

%*********************************************************
\section*{Introduction}

Let $k$ be a field. A \emph{motive over $k$} is a motive for rational
equivalence defined over $k$ with rational coefficients. The motive of
a smooth projective variety $X$ over $k$ is denoted by $\h(X)$. We
refer to \cite{Scholl} for definitions and basic properties. Our
notations will only differ from loc.\ cit.\ by the use of a
covariant set-up rather than a contravariant one. For instance, with
our conventions, $\h(\mathbb{P}^1_k) = \mathds{1} \oplus \mathds{1}(1)$.  A
motive $M$ is said to be \emph{effective} if it is isomorphic to the
direct summand of the motive of a smooth projective variety,
equivalently if it is isomorphic to a motive of the form $(X,p,n)$ for
some smooth projective variety $X$, some idempotent $p \in \End(\h(X))
:= \CH_{\dim X}(X \times X)$ and some integer $n \geq 0$. A motive $M$
is said to be of \emph{abelian type} if it is isomorphic to a motive
that belongs to the thick and rigid subcategory of motives over $k$
spanned by the motives of abelian varieties. Equivalently, $M$ is of
abelian type if one of its twist $M(n) := M \otimes \mathds{1}(n)$ is
isomorphic to the direct summand of the motive of a product of curves.
Finally, we refer to Kimura \cite{Kimura} for the notion of
\emph{finite-dimensionality} of motives. Motives of curves are
finite-dimensional and finite-dimensionality is stable under tensor
product, direct sum and direct summand. As such, motives of abelian
type are finite-dimensional. It is conjectured that all motives are
finite-dimensional. Given a finite-dimensional motive $M$, a crucial
result of Kimura \cite[Proposition 7.5]{Kimura} states that $\ker
\big(\End(M) \r \End(\overline{M}) \big)$, where $\overline{M}$
denotes the reduction modulo numerical equivalence of $M$, is a
nilpotent ideal of $\End(M)$. \medskip

This paper contains three independent sections. \medskip

\noindent \textbf{0.1.} Let $\overline{k}$ be an algebraic closure of
$k$ and let $K/k$ be an extension of $k$ which is algebraically
closed.  Our first result is the following rigidity property for
motives of abelian type.

\begin{theorem2} \label{rigidity} Let $M$ be a motive over $k$. Then
  $M_{\overline{k}}$ is of abelian type if and only if $M_K$ is of
  abelian type.
\end{theorem2}

The proof, which is given in Subsection \ref{S:rigidity}, proceeds
through a standard specialization argument. More interesting is the
following descent theorem the proof of which is given in Subsection
\ref{S:fdrigidity}.

\begin{theorem2} \label{fdrigidity} Let $M$ be a motive over $k$. Then
  $M$ is finite-dimensional if and only if $M_K$ is
  finite-dimensional.
\end{theorem2}

The main point in the proof of Theorem \ref{fdrigidity} consists,
assuming the existence of a decomposition of $M_K$ into a direct sum
of an odd-dimensional part and an even-dimensional part, in exhibiting
a similar decomposition of $M$. This is \emph{a priori} not obvious as
the size of the Chow groups of a variety might strictly increase after
base-change to a field extension. Abelian varieties (and, more
generally, motives of abelian type) are known \cite{DM} to have a
Chow--K\"unneth decomposition. As a consequence of Theorem
\ref{fdrigidity}, we obtain in Corollary \ref{C:potab} the existence
of a Chow--K\"unneth decomposition for varieties over $k$ that become
isomorphic to an abelian variety after some field extension, thereby
generalizing the case of abelian varieties which was taken care of by
Deninger--Murre \cite{DM}. \medskip

\noindent \textbf{0.2.} Let now $\Omega$ be a universal domain, that
is, an algebraically closed field of infinite transcendence degree
over its prime subfield, containing $k$.  Our second main result is
concerned with Murre's conjectures (recalled in Conjecture
\ref{Murreconj}) in the case of motives of abelian type. It is known
that any abelian variety $A$ over $k$ is endowed with a
Chow--K\"unneth decomposition which diagonalizes the induced action of
the multiplication-by-$m$ maps on the Chow groups of $A$; see Remark
\ref{R:Beauville}.  Proposition \ref{indB} shows that Murre's
conjecture (B) does not depend on the choice of a Chow--K\"unneth
decomposition for $A$, so that Beauville's vanishing conjecture
\cite{Beauville} on abelian varieties is equivalent to Murre's
conjecture (B). Proposition \ref{indC} shows that Murre's conjecture
(B) for motives of abelian type implies Murre's conjecture (C) for
motives of abelian type.  Finally, that Murre's conjecture (D) for
motives of abelian type implies Murre's conjecture (B) for motives of
abelian type relies on the main theorem of K. Xu and Z.  Xu
\cite{XuXu} which we reproduce as Theorem \ref{XX}. A combination of
Theorems \ref{T:D} and Remark \ref{R:Beauville} is the following.

\begin{theorem2} \label{D} Assume Murre's conjecture (D) for product
  of curves defined over $\Omega$. Then Murre's conjectures (A), (B),
  (C) and (D) hold for all motives of abelian type over $k$. In
  particular, Beauville's vanishing conjecture \cite{Beauville} on
  abelian varieties over $k$ holds.
\end{theorem2}

In particular, as is explained in Section \ref{S:D}, Beauville's
vanishing conjectures on abelian varieties over $k$ reduce to
Beauville's conjecture that the cycle class map $\CH_l(A_\Omega) \r
H^{2\dim A -2l}(A_{\Omega},\Q_\ell)$ be injective when
restricted to $\CH_l^{(2l)}(A_\Omega)$ for all $l$ and for all abelian
varieties $A$ over $k$.

Theorem \ref{D} confirms that conjecture (D) plays a prominent role
among Murre's conjectures. For instance, given a smooth projective
surface $S$ over $k$ with $H^1(S)=0$ and $H^2(S)$ supported on a
divisor, it is well known that Murre's conjecture (D) for $S \times S$
implies that $\CH_0(S)=\Q$.  Here $H^\bullet$ denotes any Weil
cohomology theory, e.g.\ $\ell$-adic cohomology
$H^\bullet(S_{\overline{k}},\Q_\ell)$. In particular, Murre's
conjecture (D) for fourfolds implies Bloch's conjecture on surfaces.
\medskip

\noindent \textbf{0.3.} Let $M$ be a motive over $\Omega$. Given an
adequate equivalence relation ``$\sim$'' on cycles (e.g. $\sim$ could
be trivial, numerical, homological, smash-nilpotent, algebraic, or,
when $k=\C$, Abel-Jacobi equivalence), we denote by $\CH_l(M)_\sim$
the sub-group of $\CH_l(M)$ consisting of cycles that are $\sim 0$.
In \cite{Vial1}, we proved that if $\CH_*(M)_\alg$ is generated by the
Chow groups of zero-cycles on a (possibly non-connected) curve, then
$M$ splits as a direct sum of direct summands of twisted motives of
curves. The following theorem generalizes the main result of
loc.\ cit.\ and gives a criterion for a motive to be of abelian
type.

\begin{theorem2} \label{FD} Let $\sim$ be an adequate equivalence
  relation on cycles which, when restricted to $0$-cycles, is coarser
  than albanese equivalence on $0$-cycles. Let $M$ be a motive over
  $\Omega$. Then $M$ is of abelian type if and only if $\CH_*(M)_\sim$
  is generated, via the action of correspondences, by the Chow groups
  of products of curves.
\end{theorem2}

Theorem \ref{FD} is proved in Subsection \ref{S:FD}. There we actually
state and prove Theorem \ref{factor-gen} which is the main theorem of
Section \ref{S:3} and from which is derived Theorem \ref{FD}.  Here is
an outline of the proof. Let $f:N \r M$ be a morphism of motives. By
Jannsen's semi-simplicity theorem \cite{Jannsen3}, the numerical
motive $\overline{M}$ splits as $\overline{M}_1 \oplus \overline{M}_2$
where $\overline{M}_1$ is the image of $\overline{f}$.  If $N$ is of
abelian type, then, by finite-dimensionality of $N$, this splitting
lifts to a splitting $M_1 \oplus M_2$ of $M$ such that $M_1$
isomorphic to a direct summand of $N$, and such that the induced map
$N \r M_2$ is numerically trivial and such that $\CH_*(N)_\sim \r
\CH_*(M_2)_\sim$ is surjective. Writing $M_2=(X,p,n)$, we then proceed
by induction on the dimension of $X$ to show that $M_2$ is of abelian
type. Lemma \ref{prekey} is the key lemma for that matter. Its proof
relies on a refined version of a theorem of Bloch and Srinivas
\cite{BS} which is expounded in Subsection \ref{sec-stdth}; see
Proposition \ref{BS}.  Note that when $\sim$ is the trivial
equivalence relation, Theorem \ref{FD} can be proved without using the
finite-dimensionality of the motives of curves; see Theorem
\ref{fdsummand2}.\medskip

When $M$ is the motive of a smooth projective variety $X$ over
$\Omega$, Theorem \ref{FD} can be made more precise and one need not
consider the Chow groups of $X$ in all degrees. For instance, we have
the following two theorems.

\begin{theorem2} \label{FD2} Let $\sim$ be as in Theorem \ref{FD}. Let
  $X$ be a smooth projective variety of dimension $2n$ or $2n+1$ over
  $\Omega$. Then the motive of $X$ is of abelian type if and only if
  $\CH_0(X)_\sim, \ldots, \CH_{n-1}(X)_\sim$ are generated, via the
  action of correspondences, by the Chow groups of products of curves.
\end{theorem2}

\begin{theorem2} \label{CK2} Let $\sim$ be as in Theorem \ref{FD}. Let
  $X$ be a smooth projective variety of dimension $2n-1$ or $2n$ over
  $\Omega$. Assume that $\CH_0(X)_\sim, \ldots, \CH_{n-2}(X)_\sim$ are
  generated, via the action of correspondences, by the Chow groups of
  products of curves. Then $X$ has a Chow--K\"unneth decomposition.
 \end{theorem2}

 Some applications of these two theorems are discussed in Subsections
 \ref{S:applications} and \ref{S:smash}. For instance, we consider $X$
 a smooth projective variety rationally dominated by a product of
 curves and we show that if $\dim X \leq 4$, then $X$ has a
 Chow--K\"unneth decomposition; and if $\dim X \leq 3$, then $X$ is
 finite-dimensional in the sense of Kimura. We also show, based on a
 classification result of Demailly--Peternell--Schneider \cite{DPS},
 that a complex fourfold with a nef tangent bundle has a
 Chow--K\"unneth decomposition.\medskip

 An interesting consequence of Theorem \ref{factor-gen} is that, among
 finite-dimensional motives, a motive $M$ is entirely determined, up
 to direct factors isomorphic to Lefschetz motives, by its Chow groups
 of algebraically trivial cycles. Another consequence of Theorem
 \ref{factor-gen} and its proof is the following theorem which is
 concerned with Griffiths groups. We write $\mathrm{Griff}_i(X)$ for
 $\CH_i(X)_\hom / \CH_i(X)_\alg$.

 \begin{theorem2} \label{T:Griff} Let $\sim$ be as in Theorem
   \ref{FD}. Let $f : N \r M$ be a morphism of motives. Assume that $N$
   is finite-dimensional and that there is an integer $l$ such that
   $f_* : \CH_i(N_\Omega)_\sim \r \CH_i(M_\Omega)_\sim $ is surjective
   for all $i<l$. Then $f_* : \mathrm{Griff}_i(N) \r
   \mathrm{Griff}_i(M)$ is surjective for all $i \leq l$. If moreover
   $\sim$ is algebraic equivalence, then $f_* : \CH_i(N_\Omega)_\hom
   \r \CH_i(M_\Omega)_\hom$ is surjective for all $i<l$.\qed
 \end{theorem2}
 
 In the spirit of Theorem \ref{T:Griff}, we are able to extend a
 result of R. Sebastian \cite{sebastian}. We show in Theorem
 \ref{T:smash} that if $M$ is an effective motive over $k$ such that
 $\CH_0(M_\Omega)$ is spanned via the action of correspondences by
 $0$-cycles on products of curves, then numerical equivalence agrees
 with smash-nilpotence equivalence on $1$-cycles on $M$.\medskip

 Finally, Theorem \ref{FD} can be viewed as an analogue modulo
 rational equivalence of the following result which gives yet another
 characterization of motives of abelian type.

\begin{theorem2} \label{T:abhomo} Let $X$ be a smooth projective
  complex variety, the motive of which is finite-dimensional and the
  cohomology of which is spanned, via the action of correspondences,
  by the cohomology of products of curves. Then $X$ is of abelian
  type.
\end{theorem2}

This result is essentially due to Arapura \cite{Arapura}: a
homological motive whose cohomology is spanned by the cohomology of
curves is ``motivated'' by the homological motives of curves. A
standard lifting argument for finite-dimensional motives then proves
Theorem \ref{T:abhomo}. We however include a paragraph to prove this
result as we slightly improve on Arapura's result; see \S \ref{S:hom}.

%*********************************************************
\vspace{7pt}
\section{Descent and motives of abelian type  \\
  Proof of Theorems \ref{rigidity} and \ref{fdrigidity}}

Let's first consider the following situation. Consider a scheme $X$
over an algebraically closed field $k$ such that $X_K$ is
$K$-isomorphic to an abelian variety $A$ over $K$ for some field
extension $K/k$. The $K$-isomorphism $f : X_K \r A$ is defined over a
subfield of $K$ which is finitely generated over $k$. Therefore, we
may assume that $K$ is finitely generated over $k$ and that there is a
smooth irreducible variety $U$ over $k$ with function field $K$ such
that $A$ spreads to an abelian scheme $\mathscr{A}$ over $U$ and such
that $f$ spreads to a $U$-isomorphism $f_U : X \times_k U \r
\mathscr{A}$.  Specializing at a closed point $u$ of $U$, i.e., pulling
back along the closed immersion $u \r U$, the $U$-isomorphism $f_U$
gives a $k$-isomorphism $f_u : X = X_u \r \mathscr{A}_u$. Thus, $X$ is
isomorphic to an abelian variety.

\subsection{ Proof of Theorem \ref{rigidity}.} \label{S:rigidity} Let
$k$ be an algebraically closed field and let $M$ be a motive over $k$,
say $M = (X,p,n)$. We assume that there is a field $K/k$ such that
$M_K$ is isomorphic over $K$ to a motive of abelian type. We want to
show that $M$
is of abelian type. Since $M$ becomes of abelian type over $K$, it
actually becomes of abelian type over a subfield of $K$ which is
finitely generated over $k$.  We can thus assume that $K$ is finitely
generated over $k$.
Let then $A$ be an abelian variety over $K$ such that we have a
$K$-isomorphism $M_K = (X_K,p_K,n) \cong (A,q,m)$. Up to tensoring
with the Lefschetz motive, which can be thought of as a direct summand
of the motive of an elliptic curve defined over $k$, we may assume
that $n=m=0$.
Let $Y$ be a smooth quasi-projective variety defined over $k$ with
function field $K$. Let then $U$ be a Zariski-open subset of $Y$ such
that $A$ spreads to an abelian scheme $\mathscr{A} \r U$, $q \in
\CH^{\dim A}(A \times_K A)$ spreads to a relative idempotent $\kappa
\in \CH^{\dim A}(\mathscr{A} \times_U \mathscr{A})$, and such that the
$K$-isomorphism $(X_K,p_K) \cong (A,q)$ spreads to an isomorphism
$(X_U,p_U) \cong (\mathscr{A},\kappa)$ of relative motives over $U$.
Here, $(X_U,p_U)$ denotes the constant motive over $U$ whose closed
fibers are $(X,p)$.  A Zariski-open subset of $Y$ that satisfies the
last two properties exists by the localization exact sequence for Chow
groups.

Let $t$ be a closed point of $U$. By assumption, $t$ is a smooth point
and the inclusion $j_t : t \hookrightarrow U$ is thus a regular
embedding.  Therefore, by \cite[\S 6]{Fulton}, there is a Gysin
morphism $j_t^*$ defined on Chow groups which commutes with flat
pull-backs, proper push-forwards and intersection products. It follows
that relative idempotents over $U$ specialize to idempotents and that
the relative isomorphism $(X_U,p_U) \cong (\mathscr{A},\kappa)$
specializes to an isomorphism $(X,p) = (X_t,(j_t, j_t)^*p_U ) \cong
(\mathscr{A}_t, (j_t, j_t)^* \kappa)$ defined over $k$. \qed \medskip

\subsection{Finite-dimensionality of motives is stable under descent.}
\label{S:fdrigidity} Let $k$ be a field and let $K/k$ be a field
extension. In this paragraph, we wish to study the stability under
descent of two notions attached to motives: finite-dimensionality and
Chow--K\"unneth decompositions.

\begin{definition} \label{d:fdck}
 A motive $M$ over $k$ is said to be
  \emph{finite-dimensional} if there exists a splitting $M = M^+
  \oplus M^-$ such that $S^nM^- = \Lambda^nM^+ = 0$ for $n>>0$. A
  motive $M^-$ whose symmetric powers $S^nM^-$ vanish for $n>>0$ is said to be
oddly finite-dimensional and a motive
  $M^+$ whose exterior powers $\Lambda^nM^+$ vanish for $n>>0$ is said to be
evenly finite-dimensional. The notion of finite-dimensionality is due
  independently to Kimura \cite{Kimura} and O'Sullivan.

  The motive $M$ is said to have a \emph{K\"unneth decomposition} if
  there is a finite-sum decomposition of the homological motive
  $M^\hom = \bigoplus_{i\in \Z} (M^\hom)_i$ such that $H_*((M^\hom)_i)
  = H_i(M^\hom)$.

  The motive $M$ is said to have a \emph{Chow--K\"unneth
    decomposition} if there is a finite-sum decomposition $M =
  \bigoplus_{i\in \Z} M_i$ such that this decomposition defines modulo
  homological equivalence a K\"unneth decomposition of $M$.
\end{definition}

First recall the following lemma.

\begin{lemma} \label{bc} Let $X$ be a scheme over $k$. Then the map
  $\CH_*(X) \r \CH_*(X_K)$ induced by base-change is injective.
  If $K/k$ is finite, then the composite with proper push-forward $\CH_*(X) \r
  \CH_*(X_K) \r \CH_*(X)$ is multiplication by $[K:k]$. Moreover, if $K/k$
    is a purely inseparable extension, then  $\CH_*(X)
    \r \CH_*(X_K)$ and $\CH_*(X_K) \r \CH_*(X)$ are
    both isomorphisms.  \qed
\end{lemma}
\begin{proof} This is classical; see for instance \cite[Lemma
  1.A.3]{Bloch}.  When $K/k$ is finite, the proof is immediate by
  definition of flat pull-back and proper push-forward of cycles; see
  \cite[Example 1.7.4]{Fulton}. As for the purely inseparable case,
  consider a cycle $\gamma \in \CH_*(X_K)$ defined over a finite
  purely inseparable extension $K$ of $k$ of degree $p^r$, say. The
  cycle $\frac{1}{p^r}\gamma$ is then defined over $k$. Thus, $\gamma$
  is the image of the cycle $p^r \cdot (\frac{1}{p^r}\gamma)$ under
  the map $\CH_*(X) \r \CH_*(X_K)$.
\end{proof}

Let us mention the basic fact that a motive that becomes zero after
base-change is zero.

\begin{proposition} \label{P:rigtrivial}
  Let $M$ be a motive over $k$. Then $M = 0$ if and only if $M_K = 0$.
\end{proposition}
\begin{proof}
  By definition, a motive $(X,p,n)$ is zero if and only if $p =0 \in
  \CH_*(X \times X)$. The proposition then follows from Lemma \ref{bc}.
\end{proof}

\begin{theorem} \label{fd-rigidity} Let $M$ be a motive over $k$. Then
  $M$ is finite-dimensional if and only if $M_K$ is
  finite-dimensional.
\end{theorem}
\begin{proof}
  If $M$ is finite-dimensional then it is clear that $M_K$ is
  finite-dimensional. Indeed, if $M$ splits as $M^+ \oplus M^-$ with
  $S^nM^- = \Lambda^nM^+ = 0$ for some $n$, then we have that $M_K$
  splits as $(M^+)_K \oplus (M^-)_K$ with $S^n(M^-)_K = \Lambda^n(M^+)_K =
  0$ in view of Proposition \ref{P:rigtrivial} and the fact that symmetric
powers and exterior powers of motives
      commute with base-change,. 

  Assume now that $M_K$ is finite-dimensional. This means that $M_K$
  has a splitting $(M_K)^+ \oplus (M_K)^-$ with $S^n(M_K)^- =
  \Lambda^n(M_K)^+ = 0$ for some $n>>0$. Such a splitting is defined
  over a finitely generated field over $k$ and we may assume that $K$
  is finitely generated over $k$. By a specialization argument as in
  the proof of Theorem \ref{rigidity}, we may even assume that $K$ is
  a finite extension of $k$. By Lemma \ref{bc}, we may further
  assume that $K$ is a finite Galois extension of $k$ with Galois
  group $G$, say.  Let then $p_K : = \id_{M_K} = p^+ + p^- \in
  \End(M_K)$ be the decomposition corresponding to the decomposition
  $M_K = (M_K)^+ \oplus (M_K)^-$.  The group $G$ acts on $\End(M_K)$
  as follows : for all $g \in G$ and all $f \in \End(M_K)$ we have
  $g\cdot f:= g \circ f \circ g^{-1}$.  This is well-defined since
  $p_K$ is defined over $k$ (so that $p_K \circ g = g \circ p_K$).
  Consider the $G$-invariant correspondence
  $$\tilde{p}^+ := \frac{1}{|G|} \sum_{g \in G} g\cdot p^+ \in
  \End(M_K).$$ The correspondence $p^+$ defines in $\End(M_K^\hom)$
  the projector on the even-degree homology of $M_K$ and since $M_K$
  is defined over $k$, the $\ell$-adic cohomology class of $p^+$ is
  clearly invariant under the Galois group of $k$. This yields that
  $\tilde{p}^+$ and $p^+$ are homologically equivalent.  We can thus 
  write
  $$ \tilde{p}^+ \circ \tilde{p}^+ = \tilde{p}^+ + n$$ for some
  correspondence $n \in \End(M_K)$ that is homologically trivial, and
  hence nilpotent by finite-dimensionality of $M_K$ \cite[Prop. 7.5]{Kimura}.
Since
  $\tilde{p}^+$ is $G$-invariant, it follows that $ \tilde{p}^+ \circ
  \tilde{p}^+$ is $G$-invariant and hence that $n$ is
  $G$-invariant. Looking at $\tilde{p}^+ \circ \tilde{p}^+ \circ
  \tilde{p}^+$, we see that $n \circ \tilde{p}^+ = \tilde{p}^+ \circ
  n$. Following Beilinson, we compute $$\big(\tilde{p}^+ +
  (1-2\tilde{p}^+) \circ n \big)^{ \circ 2} = \tilde{p}^+ +
  (1-2\tilde{p}^+) \circ n + n^{ \circ 2} \circ (4n-3).$$ A
  straightforward descending induction on the nilpotence index of $n$
  shows that there is a homologically trivial correspondence $m$ such
  that $q^+ := \tilde{p}^+ + m$ is a $G$-invariant idempotent in
  $\End(M_K)$.  Thus $p^+$ is homologically equivalent to the
  $G$-invariant idempotent $q^+$.  Likewise, $p^-$ is homologically
  equivalent to the $G$-invariant idempotent $q^- := \id_{M_K} -
  q^+$. By finite-dimensionality of $M_K$, it follows that $\im q^+$
  and $\im q^-$, which are motives defined over $k$, are respectively
  isomorphic over $K$ to $(M_K)^+$ and $(M_K)^-$. In particular, we
  obtain that $S^n(\im q^-) = \Lambda^n(\im q^+) = 0$ over $K$ for
  $n>>0$. By Proposition \ref{P:rigtrivial}, we conclude that $\im q^+
  \oplus \im q^-$ is a decomposition of $M$ into an evenly
  finite-dimensional motive and an oddly finite-dimensional motive.
\end{proof}

\begin{proposition}
  Let $M$ be a motive over $k$. Then $M_{\overline{k}}$ has a
  Chow--K\"unneth decomposition if and only if there exists a field
  extension $K/k$ such that $M_K$ has a Chow--K\"unneth decomposition.
\end{proposition}
\begin{proof}
  The ``only if'' part of the proposition is obvious and the ``if''
  part follows from a specialization argument as in the proof of
  Theorem \ref{rigidity} together with the compatibility of
  specialization with the cycle class map \cite[\S 20.3]{Fulton}.
\end{proof}

\begin{question} \label{Q:CK} Let $M$ be a motive over $k$. Assume
  that $M_K$ has a Chow--K\"unneth decomposition. Then does $M$ have a
  Chow--K\"unneth decomposition?
\end{question}

This question has a positive answer modulo homological equivalence, as
is observed in the following  proposition.

\begin{proposition} \label{Kuenneth} Let $M$ be a motive over $k$.
  Then $M^\hom$ has a K\"unneth decomposition if and only if $M_K$ has
  a K\"unneth decomposition.
\end{proposition}
\begin{proof}
  The ``only if'' part of the proposition is obvious. Recall that, for
  any extension $F'/F$ of algebraically closed fields and for any
  smooth projective variety $X$ over $F$, the base-change map
  $\CH_*(X) \r \CH_*(X_{F'})$ is an isomorphism modulo homological
  equivalence.  If $M_K$ has a K\"unneth decomposition, then it
  follows by specialization that, for some finite extension $l/k$,
  $M_l$ has a K\"unneth decomposition. By Lemma \ref{bc}, we may
  assume that $l/k$ is Galois. Let us write $M=(X,p,n)$ with $d=\dim
  X$. Since $M_l$ is defined over $k$, the K\"unneth projectors are
  invariant in $\End(M_l^\hom) \subseteq H^{2d}_{et}(X_{\bar{k}}
  \times_{\bar{k}} X_{\bar{k}}, \Q_\ell(d))$ under the action of the
  Galois group of $k$. It follows that $\Gal (l/k)$ acts trivially on
  those. Thus, the K\"unneth decomposition of $M_l$ is defined over
  $k$ and hence defines a K\"unneth decomposition of $M$.
\end{proof}

The following theorem shows that, assuming finite-dimensionality for
$M$, Question \ref{Q:CK} has a positive answer.

\begin{theorem} \label{t:fdck}
  Let $M$ be a motive over $k$. If $M_K$ is finite-dimensional and has
  a K\"unneth decomposition, then $M$ is finite-dimensional and has a
  Chow--K\"unneth decomposition.
\end{theorem}
\begin{proof}
 By  Theorem \ref{fd-rigidity}, $M$ is finite-dimensional; and by \cite[Prop.
7.5]{Kimura}, it follows that the kernel of $\operatorname{End}(M) \r
\operatorname{End}(M^{\hom})$ is nilpotent. Therefore, by \cite[Lemma
5.4]{Jannsen}, a sum of idempotents in $\operatorname{End}(M^{\hom})$ that adds
to the identity lifts to a sum of idempotents in $\operatorname{End}(M)$ that
adds to the identity in $\operatorname{End}(M)$. 
 Now $M$ has a K\"unneth decomposition by Proposition \ref{Kuenneth}. We thus
find
  that this K\"unneth decomposition lifts to a Chow--K\"unneth decomposition for
$M$.
\end{proof}

For instance, if $M$ is a motive such that $M_{\overline{k}}$ is of
abelian type, then $M$ has a Chow--K\"unneth decomposition. In
particular, we obtain the following result that extends the classical
result of Deninger--Murre \cite{DM} according to which every abelian
variety has a Chow--K\"unneth decomposition.

\begin{corollary} \label{C:potab} Let $X$ be a variety over $k$ such
  that $X_{\overline{k}}$ has the structure of an abelian variety.
  Then $X$ has a Chow--K\"unneth decomposition. \qed
\end{corollary}

%*********************************************************
\vspace{7pt}
\section{Murre's conjectures and motives of abelian type \\
  Proof of Theorem \ref{D}} \label{S:D}

Murre's conjectures \cite{Murre1} were originally stated for smooth
projective varieties. Here, we give a statement for motives which
contains the original statement of Murre for smooth projective
varieties.

\begin{conjecture}[Murre \cite{Murre1}] \label{Murreconj} Let $M$ be a
  motive defined over a field $k$.\medskip

(A) $M$ has a Chow--K\"unneth decomposition : $M$ splits as a finite
direct sum $\bigoplus_{i \in \Z} M_i$, where $H_*(M_i) = H_i(M)$ for
all $i$.

(B) $\CH_l(M_i) = 0$ for $i<2l$.
% and for $i>d+l$.

(C) $F^\nu \CH_l(M) := \bigoplus_{i \geq 2l+\nu} \CH_l(M_i)$ does not
depend on the choice of a Chow--K\"unneth decomposition.

(D) $\CH_l(M_{2l})_\hom = 0$ for all $l$.
\end{conjecture}

There are several remarks to be made about the formulation of Murre's
conjectures given above; see Remarks \ref{strongB}, \ref{remarkC} and
\ref{remarkD}. On the behavior of those conjectures with respect to
field extensions, we refer to Propositions \ref{basechange} and
\ref{basechange}. On the independence of conjectures (B) and (D) with
respect to the choice of a Chow--K\"unneth decomposition as in (A), we
refer to Proposition \ref{indB} for a partial answer. Finally, on
links between conjectures (B) and (C), we refer to Proposition
\ref{indC}.

\begin{remark}[On conjecture (B)] \label{strongB} Usually, conjecture
  (B) for varieties is stated in a stronger form which takes the
  following form for motives: if $M=(X,p,n)$ has a Chow--K\"unneth
  decomposition $M=\bigoplus_{i \in \Z} M_i$, then $\CH_l(M_i) = 0$
  for $i<2l$ and for $i > l -n + \dim X$. However, a combination of
  Murre's conjectures with the Lefschetz standard conjecture for $M$
  implies the strong form of conjecture (B). In particular, if
  $M=(X,p,n)$ is of abelian type, then it is known that $\CH_l(M_i) =
  0$ for $i > l -n + \dim X$.  Also, the formulation given in
  Conjecture \ref{Murreconj} has the advantage of not involving a
  variety $X$ and an integer $n$ such that $M = (X,p,n)$.
\end{remark}

\begin{remark} [On conjecture (C)] \label{remarkC} Let $G$ be the
  filtration on $\CH_l(M)$ induced by a Chow--K\"unneth decomposition
  $M= \bigoplus M'_i$. In (C), $F^\nu \CH_l(M)$ and $G^\nu \CH_l(M)$
  are meant to coincide as sub-vector spaces of $\CH_l(M)$ (not to be
  merely isomorphic, as would be the case, for instance, were $M$
  finite-dimensional).
\end{remark}

\begin{remark} [On conjecture (D)] \label{remarkD} First, for a motive
  $N$, the notation $\CH_l(N)_\hom$ is unambiguous: if $N = (Y,q,n)$,
  then $q_*\big(\CH_{l-n}(Y)_\hom \big) =
  \big(q_*\CH_{l-n}(Y)\big)_\hom$.  Indeed, the inclusion $\subseteq$
  is obvious because the action of correspondences preserves
  homological equivalence of cycles. The inclusion $\supseteq$ follows
  from the fact that $q$ is an idempotent.

  Secondly, given an integer $l$, $\CH_l(M_{2l})_\hom$ vanishes if and
  only if $\ker \big( (p_{2l})_* : \CH_l(M) \r \CH_l(M) \big) =
  \CH_l(M)_\hom$. As such, our formulation really is equivalent to
  Murre's original formulation \cite{Murre1} of conjecture (D) for
  smooth projective varieties. To see this, recall that the idempotent
  $p_{2l}$ has homology class the central projection on $H_{2l}(M)$
  and, as such, acts as the identity on $H_{2l}(M)$. Then note that,
  by functoriality of the cycle class map with respect to the action
  of correspondences, we always have $\ker \big( (p_{2l})_* : \CH_l(M)
  \r \CH_l(M) \big) \subseteq \CH_l(M)_\hom$. It is obvious that $\ker
  \big( (p_{2l})_* : \CH_l(M) \r \CH_l(M) \big) = \CH_l(M)_\hom$
  implies that $\CH_l(M_{2l})_\hom$ vanishes. Conversely,
  $\CH_l(M_{2l})_\hom = 0$ clearly implies that $\ker \big( (p_{2l})_*
  : \CH_l(M) \r \CH_l(M) \big) \supseteq \CH_l(M)_\hom$.
\end{remark}

\begin{proposition} \label{basechange} Let $K/k$ be a field extension
  and let $N$ be a motive over $k$. Assume that $N$ has a Chow--K\"unneth
  decomposition $\bigoplus_{i \in \Z} N_i$ and that $N_K$ is endowed
  with the induced Chow--K\"unneth decomposition $\bigoplus_{i \in \Z}
  (N_i)_K$. Consider the following statements :
  
 \begin{enumerate}
\item $N$ satisfies Murre's conjecture (B) ;
\item $N_K$ satisfies Murre's conjecture (B) ;
\item $N$ satisfies Murre's conjecture (D) ;
\item $N_K$ satisfies Murre's conjecture (D).
 \end{enumerate} Then $(2) \Rightarrow (1)$ and $(4) \Rightarrow (3)$.
 
 \noindent Moreover, if $k$ is a universal domain, then $(2)
 \Leftrightarrow (1)$ and $(4) \Leftrightarrow (3)$.
\end{proposition}
\begin{proof}
  That $(2) \Rightarrow (1)$ and $(4) \Rightarrow (3)$ follows
  immediately from the fact that the base-change map $\CH_l(N) \r
  \CH_l(N_K)$ is injective for all $l$; see Lemma \ref{bc}.

  Assume now that $k$ is a universal domain. Consider $F \subset k$ a
  field of definition of $N$ and of its Chow--K\"unneth decomposition
  $\bigoplus_{i \in \Z} N_i$ which is finitely generated. Let
  $\overline{K}/K$ be an algebraic closure of $K$. By the above, it is
  enough to show that if $N$ satisfies Murre's conjecture (B) or (D),
  then $N_{\overline{K}}$ endowed with the induced Chow--K\"unneth
  decomposition $\bigoplus_{i \in \Z} (N_i)_{\overline{K}}$ satisfies
  Murre's conjecture (B) or (D), respectively. Fix a field isomorphism
  $\overline{K} \cong k$ which restricts to the identity on
  $F$. Pulling back along that isomorphism, we get isomorphisms
  $\CH_l(N_i) \cong \CH_l((N_i)_{\overline{K}})$ and $\CH_l(N_i)_\hom
  \cong \CH_l((N_i)_{\overline{K}})_\hom$ for all $l$ and all
  $i$. This finishes the proof of the proposition.
  \end{proof}

  \begin{theorem}[K. Xu \& Z. Xu \cite{XuXu}] \label{XX} Assume that
    $k$ is algebraically closed. Let $X$ be a smooth projective
    variety over $k$ and let $C$ be a smooth projective curve over $k$
    with function field $K=k(C)$. Assume that $X$ has a
    Chow--K\"unneth decomposition and that $X_K$ endowed with the
    induced Chow--K\"unneth decomposition satisfies Murre's
    conjectures (B) and (D). Then $X \times C$ has a Chow--K\"unneth
    decomposition that satisfies (B).
\end{theorem}
\begin{proof}
  Let's consider a Chow--K\"unneth component $M = (X,p)$ of $\h(X)$ of
  weight $j$ such that $M_K$ satisfies Murre's conjectures (B) and
  (D).
   Consider a closed point $c$ of $C$. The idempotents $ [C \times
  \{c\}]$ and $ [\{c\} \times C]$ in $\End(\h(C))$ induce a
  Chow--K\"unneth decomposition $\h(C)= \mathds{1} \oplus \h_1(C)
  \oplus \mathds{1}(1)$, where $ \h_1(C) = (C,\pi_1)$ and $\pi_1 :=
  \Delta_C - [C \times \{c\}] - [\{c\} \times C]$.  It is clear that
  $M \otimes \mathds{1}$ and $M \otimes \mathds{1}(1)$ satisfy (B).
  The motive $M \otimes \h_1(C) = (X\times C, p \times \pi_1)$ has
  weight $j+1$ and, therefore, we only have to prove that $\CH_l(M
  \otimes \h_1(C)) = 0$ for $j+1 < 2l$ (and for $j+1 > d+l+1$, if one
  cares about the stronger form of (B); see Remark \ref{strongB}).  In
  other words, we have to show that $p \times \pi_1$ acts trivially on
  $\CH_l(X \times C)$ for $j+1 < 2l$ (and for $j+1 > d+l+1$).

  Let $\gamma \in \CH_l(X \times C)$. Since $\pi_1$ acts trivially on
  $[C] \in \CH_1(C)$, we see that $p \times \pi_1$ acts trivially on
  cycles of the form $\alpha \times [C]$ where $\alpha \in
  \CH_{l-1}(X)$. Because $$\gamma = (\gamma - \gamma|_{X\times c}
  \times [C]) + \gamma|_{X\times c} \times [C],$$ we may assume that
  $\gamma|_{X\times c} = 0$.  It is then enough to show that $p \times
  \Delta_C$ acts trivially on cycles $\gamma \in \CH_l(X \times C)$
  such that $\gamma|_{X\times c} = 0$ for $j+1 < 2l$ (and for $j+1 >
  d+l+1$).

  Let $\eta$ be the generic point of $C$. The cycle $\gamma|_{X \times
    \eta} \in \CH_{l-1}(X_K)$ is then algebraically equivalent to the
  cycle $\gamma|_{X \times c_K} \in \CH_{l-1}(X_K)$. The latter cycle
  is obtained as the image of $\gamma|_{X\times c}$ by the base-change
  map $\CH_l(X) \r \CH_l(X_K)$ and is thus zero by assumption.
  Therefore, $\gamma|_{X \times \eta} \in \CH_{l-1}(X_K)$ is
  algebraicallly trivial and hence homologically trivial. Since
  $(X_K,p_K)$ is assumed to satisfy Murre's conjecture (D), it follows
  that $ (p_K)_*(\gamma|_{X \times \eta}) = 0$ if $j = 2l-2$; and
  since $(X_K,p_K)$ is assumed to satisfy Murre's conjecture (B), it
  follows that $ (p_K)_*(\gamma|_{X \times \eta}) = 0$ if $j < 2l-2$
  (and if $j > d+l-1$).

  Now, we have the following key formula; see the proof of Lemma
  \ref{genericpoint} or \cite[ Lemma 3.2(ii)]{XuXu}.
\begin{center} $\big((p \times \Delta_C)_*\gamma\big)|_{X \times \eta}
  = (p_K)_*(\gamma|_{X \times \eta}).$ \end{center} We deduce, from
the localization exact sequence ($C^{(1)}$ denotes the set of closed
points of $C$) $$\bigoplus_{d \in C^{(1)}} \CH_l(X \times d)
\longrightarrow \CH_l(X \times C) \longrightarrow \CH_{l-1}(X \times
\eta) \longrightarrow 0,$$ that, for $j < 2l-1$ (and for $j > d+l-1$),
$$(p \times \Delta_C)_*\gamma = \sum_i \gamma_i \times
[d_i]$$ for finitely many cycles $\gamma_i \in \CH_{l}(X)$ and for
finitely many closed points $d_i$ in $C$. The correspondence $p \times
\Delta_C$ is an idempotent and this yields that $$(p \times
\Delta_C)_*\gamma = \sum_i (p_* \gamma_i) \times [d_i].$$ But then,
because $(X,p)$ is pure of weight $j$ and because $(X_K,p_K)$
satisfies Murre's conjecture (B), it follows from Proposition
\ref{basechange} that $(X,p)$ satisfies Murre's conjecture (B), i.e.,
that $p_*\CH_l(X) = 0$ for $j<2l$ (and for $j>d+l$). Thus, $(p \times
\Delta_C)_*\gamma = 0$ for $j<2l-1$ (and for $j>d+l$).
\end{proof}

\begin{proposition} \label{indB} Let $M$ be a motive over $k$ which is
  finite-dimensional. Assume that $M$ has a Chow--K\"unneth
  decomposition that satisfies Murre's conjecture (B) or (D).  Then
  any other Chow--K\"unneth decomposition of $M$ satisfies Murre's
  conjecture (B) or (D), respectively.
  
  \noindent Moreover, if $N$ is a direct summand of $M$, then $N$
  satisfies Murre's conjecture (B) or (D), respectively.
\end{proposition}
\begin{proof} Let $\bigoplus_{i \in \Z} M_i$ and $\bigoplus_{i \in \Z}
  M_i'$ be two Chow--K\"unneth decompositions for $M$. By definition
  of a Chow--K\"unneth decomposition, $M_i$ and $M_i'$ are isomorphic
  modulo homological equivalence for all $i$. Therefore, by
  finite-dimensionality, $M_i$ is isomorphic to $M_i'$ for all $i$. It
  follows that $\CH_l(M_i)= 0$ if and only if $\CH_l(M_i') = 0$, and
  that $\CH_l(M_{2l})_\hom= 0$ if and only if $\CH_l(M_{2l}')_\hom =
  0$.
  
   The Chow--K\"unneth decomposition $\bigoplus_{i \in \Z}
    M_i$ of $M$ defines a K\"unneth decomposition of $M$ modulo
    homological equivalence. The K\"unneth projectors are central in
    $\End(M^\hom)$. Therefore, as a direct summand of $M$, the motive
    $N$ has a K\"unneth decomposition.  Thus, by finite-dimensionality,
    $N$ has a Chow--K\"unneth decomposition $\bigoplus_{i \in \Z} N_i$,
    where each $N_i$ is isomorphic to a direct summand of $M_i$. This
    yields the proposition.
 \end{proof}

The following proposition and its proof are very similar to
\cite[Proposition 3.1]{Vial2}.

\begin{proposition} \label{indC} Let $M$ be a motive over $k$ which
  has a Chow--K\"unneth decomposition. Assume that Murre's conjecture
  (B) holds for $M$ and for $M \otimes M^\vee$ with respect to any
  choice of Chow--K\"unneth decomposition. Then $M$ satisfies Murre's
  conjecture (C).
\end{proposition}
\begin{proof} Let $\bigoplus_{i \in \Z} M_i$ and $\bigoplus_{i \in \Z}
  M'_i$ be two Chow--K\"unneth decompositions for $M$, and let $F$ and
  $F'$ be the induced filtrations on $\CH_l(M)$. We first show that
  $\Hom (M_i,M_j')=0$ for all $i>j$. Indeed, $(M \otimes M^\vee)_k :=
  \bigoplus_{k = j-i} M_j' \otimes M_i^\vee$ defines a Chow--K\"unneth
  decomposition for $M \otimes M^\vee$. By assumption, $\CH_0\big( (M
  \otimes M^\vee)_k \big) = 0$ for all $k<0$. Therefore,
  \begin{center} $\Hom (M_i,M_j') := \CH_0(M_j' \otimes M_i^\vee) = 0$
    for all $i>j$. \end{center} This implies that $F$ is finer than
  $F'$, i.e., that $F \subseteq F'$. By symmetry, we conclude that $F =
  F'$.
\end{proof}

We are now in a position to prove the main result of this section.

\begin{theorem} \label{T:D} Let $\Omega$ be a universal domain that
  contains $k$.  Murre's conjecture (D) for products of curves over
  $\Omega$ implies Murre's conjectures (A), (B), (C) and (D) for
  motives over $k$ which are of abelian type.
\end{theorem}
\begin{proof} Let $M$ be a motive over $k$ which is of abelian type.
  Then $M_\Omega$ is isomorphic to a direct summand of the motive of a
  product of smooth projective curves over $\Omega$.  We claim that,
  under the assumption that products of curves over $\Omega$ satisfy
  Murre's conjecture (D), products of curves over $\Omega$ satisfy
  Murre's conjectures (A)--(D). Indeed, a product of smooth projective
  curves over $\Omega$ is finite-dimensional \cite[Corollaries 4.4 \&
  5.11]{Kimura} and has a Chow--K\"unneth decomposition.  By $(3)
  \Rightarrow (4)$ in Proposition \ref{basechange}, products of curves
  defined over the function field of a curve over $\Omega$ satisfy
  Murre's conjecture (D). It follows from Theorem \ref{XX}, by a
  straightforward induction on the number of curves involved in the
  product, that Murre's conjecture (B) holds for products of curves
  over $\Omega$.
    
    Now, by
  Proposition \ref{indB}, we find that $M_\Omega$ satisfies Murre's
  conjectures (A), (B) and (D). A combination of Propositions
  \ref{basechange} and \ref{indB} shows that $M$ satisfies Murre's
  conjectures (A), (B) and (D). Finally, the motive $M \otimes M^\vee$ is
  also of abelian type over $k$. Therefore Murre's conjecture (B) also
  holds for $M \otimes M^\vee$. Thus, thanks to Proposition
  \ref{indC}, $M$ also satisfies Murre's conjecture (C).
\end{proof}

\begin{remark} \label{R:Beauville}
Let $A$ be an abelian variety of dimension $d$ over $k$. Let $m$ be an
integer and let $[m] : A \r A$ denote the multiplication-by-$m$
endomorphism of $A$. Then there exists a Chow--K\"unneth decomposition
$\{\Pi_i\}$ for $A$ such that $\Pi_i \circ \Gamma_{[m]} = \Gamma_{[m]}
\circ \Pi_i = m^i \cdot \Pi_i \in \CH_d(A \times A)$; see \cite{DM}.
Moreover, there is a decomposition \medskip \begin{center} $\CH_l(A) =
  \bigoplus_{l}^{l+d} \CH_l^{(i)}(A)$, where $\CH_l^{(i)}(A):= \{\alpha
  \in \CH_l(A) \, | \, [m]_*\alpha=m^i\alpha, \forall m \in \Z \}.$
\end{center} \medskip Beauville \cite{Beauville} conjectured that
$\CH_l^{(i)}(A) = 0$ for $i < 2l$, and Murre checked \cite[Lemma
2.5.1]{Murre1} that
$$(\Pi_i)_*\CH_l(A) = \CH_l^{(i)}(A).$$ We thus see, thanks to
Proposition \ref{indB}, that Beauville's conjecture for $A$ equipped
with the Chow--K\"unneth decomposition $\{\Pi_i\}$ above is equivalent
to Murre's conjecture (B) for $A$ equipped with any Chow--K\"unneth
decomposition.

Beauville \cite{Beauville} also conjectured that the cycle class map
$\CH_l(A) \r H^{2\dim A - 2l}(A_{\overline{k}},\Q_\ell)$ to
$\ell$-adic cohomology is injective when restricted to
$\CH_l^{(2l)}(A)$. The identity $(\Pi_i)_*\CH_l(A) = \CH_l^{(i)}(A)$
shows that this conjecture is actually equivalent to Murre's
conjecture (D) for $A$. Since motives of abelian type are spanned
either by motives of curves or by motives of abelian varieties,
Proposition \ref{indB} implies that Murre's conjecture (D)
for products of curves is equivalent to Beauville's conjecture that
$\CH_l^{(2l)}(A) \r H^{2\dim A - 2l}(A_{\overline{k}},\Q_\ell)$ is
injective for all $l$ and all abelian varieties.

We have thus showed that if $\CH_l^{(2l)}(A_\Omega) \r H^{2\dim A -
  2l}(A_{\Omega},\Q_\ell)$ is injective for all $l$ and all abelian
varieties $A$ over $k$, then Beauville's vanishing conjecture holds,
i.e., $\CH_l^{(i)}(A) = 0$ for all abelian varieties $A$ over $k$ and
all $i < 2l$.
\end{remark}

%*********************************************************
\vspace{7pt}
\section{Chow groups and motives of abelian type \\
  Proof of Theorem \ref{FD}} \label{S:3}

\subsection{Chow groups and field extensions}

The following lemma is certainly well known.

\begin{lemma} \label{lemma2} Let $f : M \r N$ be a morphism of motives
  defined over $k$. Assume that, for some field extension $K/k$,
  $(f_K)_* : \CH_0(M_K) \r \CH_0(N_K)$ is surjective or injective.  Then
  $f_* : \CH_0(M) \r \CH_0(N)$ is surjective or injective, respectively.
\end{lemma}
\begin{proof} Given a smooth projective variety $X$ over $k$, by Lemma
  \ref{bc}, the base-change map $\CH_l(X) \r \CH_l(X_K)$ is injective
  for all $l$. Besides, we have the following commutative diagram
  \begin{center}
    $\xymatrix{\CH_0(M) \ar[r] \ar[d]^{f_*} & \CH_0(M_K)
      \ar[d]^{(f_K)_*}
      \\
      \CH_0(N) \ar[r] & \CH_0(N_K). }$
  \end{center}It is then straightforward to see that if $(f_K)_*$ is
  injective, then $f_*$ is injective.

  Let's now assume that there is a field $K/k$ such that $(f_K)_* :
  \CH_0(M_K) \r \CH_0(N_K)$ is surjective. Let's pick a cycle $\gamma
  \in \CH_0(N)$ and let's denote by $\gamma_K$ its image in
  $\CH_0(N_K)$. There is a cycle $\beta \in \CH_0(M_K)$ such that
  $(f_K)_* \beta = \gamma_K$. The cycle $\beta$ is defined over a
  finitely generated extension of $k$. We may therefore assume that
  $K$ is finitely generated over $k$ and that it is the function field
  of a smooth quasi-projective variety $Y$ over $k$. Let $y$ be a
  closed point in $Y$ and let $k(y)/k$ be its residue field. Such a
  point defines a regular embedding $y \hookrightarrow Y$, so that,
  for any smooth projective variety $X$ over $k$ and for any integer
  $l$, there is a specialization map $\sigma : \CH_{l}(X_K) \r
  \CH_l(X_{k(y)})$ which commutes with flat pull-backs, proper
  push-forwards and intersection product; see \cite[\S \S 6 \&
  20.3]{Fulton}. Moreover, for a cycle $\alpha \in \CH_l(X)$, we have
  $\sigma(\alpha_K) = \alpha_{k(y)}$ because $\sigma(\alpha_K)$ is
  obtained as the intersection of $\alpha \times Y$ with $X \times
  k(y)$. It immediately follows, after specialization, that
  $(f_{k(y)})_* \sigma(\beta) = \alpha_{k(y)}$.  Now, for any smooth
  projective variety $X$ over $k$, the composite map $\CH_0(X) \r
  \CH_0(X_{k(y)}) \r \CH_0(X)$ is multiplication by $[{k(y)}:k]$; see
  Lemma \ref{bc}.  These maps commute with the action of
  correspondences and this yields that $\alpha$ lies in the
  image of $f_* : \CH_0(M) \r \CH_0(N)$.
 \end{proof}

  The converse to Lemma \ref{lemma2} is not true in general. Consider
  for instance a smooth projective curve $C$ over a finite field $F$
  with positive genus, a closed point $c$ on $C$, and the
  correspondence $f := [C \times c] \in \CH_1(C \times C) =
  \End(\h(C))$. Then $f_* : \CH_0(C) \r \CH_0(C)$ is an isomorphism
  (both Chow groups are spanned by $[c]$). However, $(f_K)_* :
  \CH_0(C_K)_\hom \r \CH_0(C_K)_\hom$ is zero for all field extensions
  $K/F$, and $\CH_0(C_K)_\hom \neq 0$ for some extension $K/F$ (for
  instance, a finite extension of $F(t)$ over which
  $\mathrm{Pic}^0(C_K)$ acquires a non-torsion rational point). Thus,
  for such a choice of field $K$, $(f_K)_* : \CH_0(C_K) \r \CH_0(C_K)$
  is neither injective, nor surjective.

  Anyhow, when the base-field $k$ is a universal domain, the converse
  does hold:
 
\begin{lemma} \label{lemma} Let $f : M \r N$ be a morphism of
    motives defined over $\Omega$. Then, for all fields $K$ over
    $\Omega$, $f_* : \CH_0(M) \r \CH_0(N)$ is surjective or injective
    if and only if $(f_K)_* : \CH_0(M_K) \r \CH_0(N_K)$ is surjective
    or injective, respectively.
\end{lemma}
\begin{proof}
  Let us first prove the lemma for fields $K$ that have same
  cardinality as $\Omega$.  Let $k \subset \Omega$ be a field of
  definition of $f : M \r N$ which is finitely generated. Let
  $\overline{K}$ be an algebraic closure of $K$ and fix a field
  isomorphism $\sigma : \overline{K} \stackrel{\simeq}{\rightarrow}
  \Omega$ which restricts to the identity on $k$. We have the
  following commutative diagram
  \begin{center}
    $\xymatrix{\CH_0(M) \ar[r] \ar[d]^{f_*} & \CH_0(M_K) \ar[r]
      \ar[d]^{(f_K)_*}
      & \CH_0(M_{\overline{K}})  \ar[d]^{(f_{\overline{K}})_*} \\
      \CH_0(N) \ar[r] & \CH_0(N_K) \ar[r] & \CH_0(N_{\overline{K}})}$
  \end{center}
  where the horizontal arrows are induced by base-change and are
  therefore injective by Lemma \ref{bc}. We note, by pulling back
  along $\sigma$ or $\sigma^{-1}$, that $f_*$ is surjective or
  injective if and only if $(f_{\overline{K}})_*$ is surjective or
  injective, respectively. The lemma then follows from Lemma
  \ref{lemma2}.

  Now, assume that $K/\Omega$ is any field extension and that $f_*$ is
  surjective or injective, respectively. Let $\alpha$ be a cycle in
  $\CH_*(N_K)$ and let $\beta$ be a cycle in $\ker (f_K)_*$. These
  cycles are defined over a subfield $L$ of $K$ which is finitely
  generated over $\Omega$. By the above, $(f_L)_*$ is surjective or
  injective, respectively. It is then straightforward to see that
  $\alpha \in \im (f_K)_*$ and that $\beta = 0$. Thus, $(f_K)_*$ is
  surjective or injective, respectively.
\end{proof}

\subsection{A refinement of a theorem of Bloch and Srinivas}
\label{sec-stdth}

In order to prove the key Lemma \ref{prekey}, we need a slight
refinement of the decomposition of the diagonal argument of Bloch and
Srinivas which appears in \cite{BS}. This is embodied in Proposition
\ref{BS}. \medskip

First we prove a lemma which seems to be known as Lieberman's lemma
and which is quoted in \cite[3.1.4]{An} and \cite[1.10]{Scholl}
without proof. Let $X$, $X'$, $Y$ and $Y'$ be smooth projective
varieties over a field $k$.  Let $a \in \CH_p(X' \times X)$, $b \in
\CH_q(Y \times Y')$ and $\gamma \in \CH_r(X \times Y)$ be
correspondences. Let's write $({}^t a,b) := \tau_*({}^ta \times b)$,
where $\tau : X \times X' \times Y \times Y' \r X \times Y \times X'
\times Y'$ is the map permuting the two middle factors.

\begin{lemma} \label{BS-lemma} We have the formula
  $$({}^t a,b)_*\gamma =  b \circ \gamma \circ a
  \ \in \CH_{p+q+r- \dim X -\dim Y}(X' \times Y').$$
\end{lemma}
\begin{proof}
  The lemma is proved in \cite[16.1.1]{Fulton} in the case when $a$
  and $b$ are either graphs of morphisms or the transpose thereof.  We
  reduce to this case by showing that every correspondence is the sum
  of the composite of graphs of morphisms and their transpose.

%  If $X_1, X_2, \ldots, X_n$ are varieties and if $1 \leq i_1 < \cdots
%  < i_m \leq n$ are integers, let's write $p^{X_1X_2\ldots
%    X_n}_{X_{i_1}X_{i_2}\cdots X_{i_m}}$ for the projection $X_1
%  \times X_2 \times \cdots \times X_n \r X_{i_1} \times X_{i_2} \times
%  \cdots \times X_{i_m}$.

 If $X_1, X_2, X_3$ are varieties, let's write, for $i \in \{1,2\}$,
$p^{X_1X_2}_{X_{i}}$ for the projection $X_1\times X_2  \r X_{i}$; let's also
write, for $i \neq j \in \{1,2,3\}$, $p^{X_1X_2X_3}_{X_{i}X_j}$ for the
projection $X_1\times X_2\times X_3  \r X_{i}\times X_j$.

  Let's consider the case $X=X'$ and $b = [W]$, where $W$ is an
  irreducible subvariety of $Y \times Y'$ of dimension $q$. We are
  going to prove that $b$ is the composite of the graph of a morphism
  with the transpose of the graph of another morphism.  For this
  purpose, let's consider an alteration $\sigma : \widetilde{W} \r W$
  and let's define the composite morphism $h : \widetilde{W} \r W
  \hookrightarrow Y \times Y'$. Then, by proper pushforward, we have
  $\deg (\sigma) \cdot b = h_*[\widetilde{W}]$. Now,

  \begin{center}
    $\begin{array} {lcl} \deg (\sigma) \cdot b \circ \gamma & = &
      (p_{XY'})_*\big((p_{YY'}^*h_*[\widetilde{W}]) \cap
      p_{XY}^*\gamma \big)
      \\
      & = & (p_{XY'})_*\big((\id_X \times h)_*(p_{Y'}^{XY'})^*
      [\widetilde{W}] \cap
      p_{XY}^*\gamma \big) \\
      & = & (p_{XY'})_*\big((\id_X \times h)_*(\id_X \times h)^*
      p_{XY}^*\gamma \big)\\
      & = & (\id_X \times (p_{Y'}^{YY'} \circ h))_* (\id_X \times
      (p_Y^{YY'}
      \circ h))^* \gamma \\
      & = & (\Delta_X, \Gamma_{p_{Y'}^{YY'} \circ h})_* (\Delta_X,
      {}^t \Gamma_{p_{Y}^{YY'} \circ h})_* \gamma  \\
      & = & \Gamma_{p_{Y'}^{YY'} \circ h} \circ {}^t
      \Gamma_{p_{Y}^{YY'} \circ h} \circ \gamma.
  \end{array}$
  \end{center}
  Here, we have omitted the superscript ``$XYY'$''. The first equality
  is by definition of the composition law for correspondences. The
  second equality follows from the fibre square
  \begin{center}
    $\xymatrix{ X \times W \ar[d]_{p_W^{XW}} \ar[r]^{\id_X \times h \
        \ } & X \times Y \times Y' \ar[d]_{p_{YY'}} \\ W \ar[r]^h & Y
      \times Y'.}$
  \end{center}
  The third equality follows from the projection formula. The fourth
  equality follows from the equalities $p_{XY} \circ (\id_X \times h) =
  \id_X \times (p_Y^{YY'} \circ h)$ and $p_{XY'} \circ (\id_X \times h)
  = \id_{X} \times (p_{Y'}^{YY'} \circ h)$. The fifth equality is
  \cite[16.1.2.(c)]{Fulton}. Finally, the last equality follows from
  \cite[16.1.1]{Fulton}.

  This last equality holds for all smooth projective varieties $X$,
  all integers $r$ and all correspondences $\gamma \in \CH_r(X \times
  Y)$. By Manin's identity principle \cite[4.3.1]{An}, we get
  \begin{equation} % \label{eq1} 
\deg (\sigma) \cdot b =
    \Gamma_{p_{Y'}^{YY'} \circ h} \circ {}^t \Gamma_{p_{Y}^{YY'} \circ
      h}. \nonumber
  \end{equation}

  Thus, as claimed, every correspondence is the sum of the composite
  of graphs of morphisms and their transpose.
\end{proof}

\begin{lemma} \label{genericpoint} Let $X$ and $Y$ be smooth
  projective varieties over a field $k$. Let $\Gamma \in \CH_n(X \times
  Y)$ be a correspondence. Let $\eta_X$ be the generic point of $X$
  and let $[\eta_X] \in \CH_0(X_{k(X)})$ be the class of $\eta_X$
  viewed as a rational point of $X_{k(X)}$.  Then, under the natural
  map $\CH_n(X \times Y) \r \CH_{n-\dim X} (k(X) \times Y)$, $\Gamma$ is
  mapped to $(\Gamma_{k(X)})_*[\eta_X]$.
\end{lemma}
\begin{proof} Let $d:=\dim X$.  Since the map $\CH_{d}(X \times X) \r
  \CH_{0}(k(X) \times X)$ is obtained as the direct limit, ranging
  over the open subsets $U$ of $X$, of the flat pullback maps
  $\CH_{d}(X \times X) \r \CH_{0}(U \times X)$, we see that
  $[\Delta_X] \in \CH_d(X \times X)$ is mapped to $[\eta_X] \in
  \CH_{0}(k(X) \times X)$. Besides, by Lemma \ref{BS-lemma},
  $(\Delta_X, \Gamma)_* \Delta_X = \Gamma \circ \Delta_X =
  \Gamma$. The lemma then follows by commutativity, for all integer
  $r$, of the diagram
    \begin{center}
      $ \xymatrix{\CH_{r}(X \times X) \ar[d]^{(\Delta_X, \Gamma)_*}
        \ar[r] & \CH_{r-d}(k(X) \times X) \ar[d]^{(\Gamma_{k(X)})_*}
        \ar[r] &
        0 \\
        \CH_{n+r-d}(X \times Y) \ar[r] & \CH_{n+r-2d}(k(X) \times Y)
        \ar[r] & 0.}$
    \end{center} Let us prove commutativity of the diagram. It is
    obtained as the direct limit over the open inclusions $j_U : U
    \hookrightarrow X$ of the diagrams
    \begin{center} $ \xymatrix{\CH_{r}(X \times X) \ar[d]^{(\Delta_X,
          \Gamma)_*} \ar[rr]^{({}^t\Gamma_{j_U},\Delta_X)_*}& &
        \CH_{r}(U
        \times X) \ar[d]^{(\Delta_U,\Gamma)_*} \\
        \CH_{n+r-d}(X \times Y)
        \ar[rr]^{({}^t\Gamma_{j_U},\Delta_Y)_*} & & \CH_{n+r-d}(U
        \times Y).}$
    \end{center} The action of $(\Delta_U,\Gamma)$ on $ \CH_{r}(U
    \times X)$ is well-defined because $(\Delta_U,\Gamma)$ has a
    representative whose support is proper over $U \times Y$,
    cf. \cite[Remark 16.1]{Fulton}. These diagrams commute for all $U$
    for the following reason. Let $U$, $V$ and $W$ be nonsingular open
    varieties and let $\alpha \in \CH_i(U \times V)$ (resp. $\beta \in
    \CH_j(V \times W)$) be a correspondence which has a representative
    which is proper over $U$ and $V$ (resp. $V$ and $W$). Then, by
    loc.\ cit.\ , it is possible to define the composite $\beta
    \circ \alpha$ and to show as in \cite[Proposition
    16.1.2(a)]{Fulton} that $(\beta \circ \alpha)_* = \beta_* \circ
    \alpha_*$ on cycles. We may now conclude that the diagram is
    commutative by checking that $(\Delta_U,\Gamma) \circ
    ({}^t\Gamma_{j_U},\Delta_X) = ({}^t\Gamma_{j_U},\Delta_Y) \circ
    (\Delta_X, \Gamma) = ({}^t\Gamma_{j_U},\Gamma)$.
\end{proof}

\begin{proposition} \label{BS} Let $M =(X,p)$ and $N = (Y,q,n)$ be two
  motives over a field $k$ and let $\varphi : N \r M$ be a morphism.
  Suppose that $(\varphi_{k(X)})_* : \CH_0(N_{k(X)}) \r \CH_0(M_{k(X)})$
  is surjective. Then there exist a morphism $\Gamma_1 : M \r N$, a
  smooth projective variety $Z$ of dimension $\dim X -1$ over $k$ and
  a morphism $\Gamma_2 : M \r M$ that factors through the motive $(Z,
  \Delta_Z, 1)$ such that $$p = \varphi \circ \Gamma_1 + \Gamma_2.$$
\end{proposition}
\begin{proof}
   As in the proof of Lemma \ref{genericpoint}, we have the following
  commutative diagram whose rows are exact by localization.
  \begin{center}
    $ \xymatrix{\CH_{d-n}(X \times Y) \ar[d]^{(\Delta_X, \varphi)_*}
      \ar[r] & \CH_{-n}(k(X) \times Y) \ar[d]^{(\varphi_{k(X)})_*}
      \ar[r] &
      0 \\
      \CH_d(X \times X) \ar[r] & \CH_0(k(X) \times X) \ar[r] & 0.}$
  \end{center}
  Let $\eta_X$ be the generic point of $X$ and view it as a
  rational point of $k(X) \times X$ over $k(X)$. Then $p \in \CH_d(X
  \times X)$ maps to $(p_{k(X)})_*[\eta_X] \in \CH_0(k(X) \times X)$
  by Lemma \ref{genericpoint}.  Because $(\varphi_{k(X)})_* :
  \CH_0(N_{k(X)}) \r \CH_0(M_{k(X)})$ is surjective, there exists $y
  \in \CH_{-n}(k(X) \times Y)$ such that $(\varphi_{k(X)})_*y =
  (p_{k(X)})_*[\eta_X]$. Let $\alpha$ be a lift of $y$ in $\CH_{d-n}(X
  \times Y)$. By commutativity of the diagram, we have that $\Gamma_2
  := p - (\Delta_X, \varphi)_*\alpha$ maps to zero in $\CH_0(k(X)
  \times X)$.  Therefore, by the localization exact sequence for Chow
  groups, $\Gamma_2$ is supported on $D \times X$ for some
  codimension-one closed subscheme $D$ of $X$, i.e., there is a $\beta
  \in \CH_d(D \times X)$ that maps to $\Gamma_2$. Let's write $\iota :
  D \hookrightarrow X$ for the inclusion map. Consider then $\sigma : Z \r D$ 
  an alteration of $D$, that is, $\sigma$ is a generically finite morphism with
$Z$ smooth. Such a morphism exists for any variety $D$ over $k$ by de Jong's
alteration theorem. Using the alteration $\sigma : Z \r D$, we see that there is
a cycle $\gamma
  \in \CH_d(Z \times X)$ that maps to $\Gamma_2$ under the natural map
  $\CH_d(Z \times X) \r \CH_d(X \times X)$. By Lemma \ref{BS-lemma},
  we then have $\Gamma_2 = ((\iota \circ \sigma) \times \id_X)_*
  \gamma = \Gamma_{\iota \circ \sigma} \circ \gamma$, so that
  $\Gamma_2$ factors through $(Z,\Delta_Z,1)$. By Lemma \ref{BS-lemma}
  again, we have $(\Delta_X, \varphi)_*\alpha = \varphi \circ \alpha$,
  so that if we set $\Gamma_1 := q \circ \alpha \circ p$, then we have
  $p = \varphi \circ \Gamma_1 + \Gamma_2.$
\end{proof}

\subsection{Three lemmas} \label{splitting-app} The following three
lemmas are the building blocks to the proof of  Theorem
\ref{factor-gen} which is the main theorem of Section \ref{S:3}.\medskip

Given a motive $M$, we denote $\overline{M}$ its reduction modulo
numerical equivalence. Recall the following.  Let $f : N \r M$ be a
morphism of motives and let $\overline{f} : \overline{N} \r
\overline{M}$ be its reduction modulo numerical equivalence. If $N$ is
finite-dimensional and if $\overline{f}$ has a left-inverse, then $f$
has a left-inverse. Indeed, consider a morphism $h : M \r N$ such that
$\overline{h} \circ \overline{f} = \id_{\overline{N}}$. By
\cite[Proposition 7.5]{Kimura}, $h \circ f - \id_N \in \End(N)$ is
nilpotent. It is then clear that $f$ has a left-inverse. A similar
statement holds if $M$ is finite-dimensional and if ``left-inverse''
is replaced with ``right-inverse''.

\begin{lemma} \label{split} Let $f : N \r M$ be a morphism of motives
  defined over a field $k$.  Assume that $N$ is finite-dimensional in
  the sense of Definition \ref{d:fdck}. Then $M$ splits as $M_1 \oplus M_2$,
where the
  induced morphism $N \r M_1$ has a right-inverse and where the
  induced morphism $N \r M_2$ is numerically trivial.
\end{lemma}
\begin{proof}
  The morphism $f : N \r M$ reduces modulo numerical equivalence to a
  morphism $\overline{f} : \overline{N} \r \overline{M}$. By Jannsen's
  semi-simplicity Theorem \cite{Jannsen3}, $\overline{N}$ admits a
  splitting $\overline{N}_1 \oplus \overline{N}_2$ and $\overline{M}$
  admits a splitting $\im \overline{f} \oplus \overline{M}'$ such that
  $\overline{f}$ induces an isomorphism $\overline{N}_1 \r \im
  \overline{f}$ and such that $\overline{N}_2 \r \overline{M}$ and
  $\overline{N} \r \overline{M}'$ are zero. By finite-dimensionality
  of $N$, $\overline{N}_1$ lifts to a direct summand $N_1$ of $N$. Let
  $s : N_1 \r M$ be the restriction of $f$ to $N_1$ and let
  $\overline{r} : \overline{M} \r \overline{N}_1$ be a left-inverse to
  $\overline{s}: \overline{N}_1 \r \overline{M}$.  By
  finite-dimensionality of $N_1$, we find a morphism
  $r : M \r N_1$ which is a left-inverse to $s : N_1 \r M$.  The
  idempotent $s \circ r : M \r M$ thus defines a direct summand $M_1$
  of $M$ which is isomorphic to $N_1$ and it is clear that the induced
  morphism $N \r M_1$ has a right-inverse. We can therefore write $M =
  M_1 \oplus M_2$, where $M_2$ is defined by the idempotent
  $\id_M-s\circ r$. That the induced morphism $N \r M_2$ is
  numerically trivial follows from the facts that $\overline{f}$ is
  zero on $\overline{N}_1$ and that $\overline{r} \circ \overline{s} =
  \id_{\overline{N}_1}$.
\end{proof}

For convenience, for a motive $M$, when we say that there
exist a smooth projective variety $Z$ of dimension at most a negative
integer over $k$ and an idempotent $q \in \End(\h(Z))$ such that $M
\cong (Z,q,l)$, we mean that $M=0$.

\begin{lemma} \label{induction} Let $M=(X,p,n)$ be a motive over $k$
  and let $l$ be an integer greater than $n$ such that
  $\CH_i(M_{\Omega})=0$ for all $i <l$. Then there exist a smooth
  projective variety $Z$ of dimension at most $\dim X -l +n$ over $k$
  and an idempotent $q \in \End(\h(Z))$ such that $M \cong (Z,q,l)$.
\end{lemma}
\begin{proof} This is due to Kahn--Sujatha \cite{KahnSujatha}. Let's
  however give a proof. Clearly, we may assume $n=0$. Let's proceed by
  induction on $l$. If $l\leq 0$, there is nothing to prove.  Assume
  that $\CH_0(M_{\Omega}) = 0$ and hence $\CH_0(M_{k(X)}) = 0$. Then,
  by Proposition \ref{BS}, there is a smooth projective variety $W$ of
  dimension $\dim X - 1$ and an idempotent $r$ such that $M$ is
  isomorphic to $(Z,r,1)$.  Assume now that $\CH_i(M_{\Omega})=0$ for
  all $i <l$ (and hence that $\CH_i(M_{k(X)})=0$ for all $i <l$) and
  that there is a smooth projective variety $Y$ of dimension $\dim X -
  l +1$ and an idempotent $q_Y$ such that $M$ is isomorphic to
  $(Y,q_{Y},l-1)$. The motive $(Y,q_{Y},l-1)$ is such that
  $\CH_0((Y,q_{Y})_\Omega) = 0$ and hence $\CH_0((Y,q_{Y})_{k(Y)}) =
  0$. The case $l=1$ has been settled above, and we may thus conclude
  to the existence of a smooth projective variety $Z$ of dimension at
  most $\dim X -l$ and of an idempotent $q$ such that $M \cong
  (Z,q,l)$.
\end{proof}

\begin{remark} \label{effective}
  Lemma \ref{induction} gives a necessary and sufficient condition for
  a motive $M$ over $k$ to be effective, namely that
  $\CH_l(M_\Omega)=0$ for all $l <0$. Indeed, after choosing an
  embedding $k(X) \hookrightarrow \Omega$, we see that
  $\CH_l(M_\Omega)=0$ for all $l <0$ implies that $\CH_l(M_{k(X)})=0$
  for all $l <0$. It then follows from Lemma \ref{induction} that $M$
  is effective. Conversely, if $M$ is effective, then $M_\Omega$ is
  also effective. It is then clear that $\CH_l(M_\Omega)=0$ for all $l
  <0$.
\end{remark}

\begin{remark}
  Further to Remark \ref{effective}, Lemma \ref{induction} gives a
  necessary and sufficient condition for a motive $M$ over $k$ to be
  isomorphic to a direct summand of a twisted motive of a variety of
  dimension at most $d$. Given a motive $M$ over $k$, there exist a
  smooth projective variety $X$ of dimension at most $d$ and an
  integer $l$ such that $M$ is isomorphic to a direct summand of
  $\h(X)(l)$ if and only if there exist integers $l$ and $l'$ with
  $-l-l'=d$ such that $\CH_i(M_\Omega)=0$ for all $i<l$ and
  $\CH_j(M^\vee_\Omega)=0$ for all $j<l'$. The condition is indeed
  clearly sufficient and it is necessary by the following. That
  $\CH_i(M_\Omega)=0$ for all $i<l$ implies that $M$ is isomorphic to
  a motive of the form $(Y,q,l)$ for some smooth projective variety
  $Y$ over $k$. That $\CH_j(M^\vee_\Omega)= \CH_{j+l+\dim
    Y}(Y_\Omega,q_\Omega) = 0$ for all $j<l'$ implies that $M^\vee$ is
  isomorphic to a motive of the form $(X,p,l+l'+\dim Y)$ for some
  smooth projective variety $X$ of dimension at most $-l-l'$ over $k$
  and some idempotent $p \in \End(\h(X))$.  Dualizing this isomorphism gives $M
\cong
  (X,{}^tp,l)$.
\end{remark}

\begin{lemma} \label{prekey} Let $f : N \r M$ be a morphism of motives
  defined over a field $k$ with $M=(X,p)$.  Assume that $f$ is
  numerically trivial and that $N$ is finite-dimensional in the sense
  of Definition \ref{d:fdck}. If $(f_{k(X)})_* : \CH_0(N_{k(X)}) \r
\CH_0(M_{k(X)})$ is
  surjective, then $\CH_0(M_K) = 0$ for all field extensions $K/k$.
\end{lemma}
\begin{proof}
  By Proposition \ref{BS}, we get the existence of $\Gamma_1 \in
  \Hom(M,N)$, and of $\Gamma_2 \in \End(\h(X))$ which factors through
  $(Z,\Delta_Z,1)$ for some variety $Z$, such that $p = \gamma \circ
  \Gamma_1 + \Gamma_2$. The arguments below work equally well after
  base change to $K$ and, without loss of generality, we therefore
  assume $K=k$.  The action of $\Gamma_2$ on $\CH_0(X)$ factors
  through $\CH_0(Z,\Delta_Z,1) = \CH_{-1}(Z) =0$, so that $\Gamma_2$
  acts trivially on $\CH_0(X)$. Therefore, $p_*x = (\gamma \circ
  \Gamma_1)_*x$ for all $x \in \CH_0(X)$. Since $p$ is an idempotent,
  we also get $p_*x = (\gamma \circ \Gamma_1)^{\circ n}_*x$ for all
  positive integers $n$ and all $x \in \CH_0(X)$. By assumption,
  $\gamma$ is numerically trivial.  Hence, $\Gamma_1 \circ \gamma \in
  \End(N)$ is numerically trivial.  Now, $N$ is finite-dimensional and
  \cite[Proposition 7.5]{Kimura} yields that $\Gamma_1 \circ \gamma$
  is nilpotent.  We have thus proved that $p_*x = 0$ for all $x \in
  \CH_0(X)$, i.e.,  that $\CH_0(M) = 0$.
\end{proof}

\subsection{Proof of the main theorem} \label{S:FD}

Let $X$ be a smooth projective variety over $k$.  Two $0$-cycles
$\alpha$ and $\beta$ on $X$ are said to be \emph{albanese equivalent}
if $\alpha - \beta$ has degree zero and lies in the kernel of the
albanese map $\CH_0(X) \r \mathrm{Alb}(X)(k) \otimes \Q$. From now on,
$\sim$ denotes an adequate equivalence relation on algebraic cycles
which is, when restricted to $0$-cycles, coarser than albanese
equivalence on $0$-cycles.  For instance, $\sim$ could be homological
equivalence, algebraic equivalence, smash-nilpotent equivalence,
Abel-Jacobi equivalence (when $k=\C$) or any intersection thereof. By
Remark \ref{remarkD} applied to $\sim$ instead of homological
equivalence, given a motive $M$, we may consider cycles that are $\sim
0$ on $M$ and the notation $\CH_i(M)_\sim$ is unambiguous.

\begin{theorem} \label{factor-gen} Let $f : N \r M = (X,p,n)$ be a
  morphism of motives over $k$. Assume that \begin{itemize}
  \item there exists an integer $l$ such that the induced maps
    $(f_\Omega)_* : \CH_i(N_\Omega)_{\sim} \r \CH_i(M_\Omega)_\sim$
    are surjective for all $i < l$;
\item $N$ is finite-dimensional.
\end{itemize}
 Then $M$ splits as $Q
  \oplus R(l)$, where 
  \begin{itemize}
\item  $Q$ is isomorphic to a direct summand of $N
  \oplus \bigoplus_{i=n}^{l-1}\h(C_i)(i)$ for some curves $C_n,\ldots
  , C_{l-1}$;
\item $R$ is isomorphic to a direct summand of $\h(Z)$ for some smooth
  projective variety $Z$ of dimension at most $ \dim X - l + n$ over
  $k$.
  \end{itemize}
  If, in addition $M \cong M^\vee(d)$ for some $d$ (e.g.\ $M=\h(X)$
  with $d=\dim X$), then $M$ splits as $Q' \oplus R'(l)$, where $Q'$
  is isomorphic to a direct summand of $N \oplus N^\vee(d) \oplus
  \bigoplus_i \h(C_i)(i)$ for some curves $C_i$, and where $R'$ is
  isomorphic to a direct summand of $\h(Z')$ for some smooth
  projective variety $Z'$ of dimension at most $d - 2l$ over $k$.
\end{theorem}

\begin{proof}
  Up to working with each irreducible component of $X$ separately, we
  may assume that $X$ is irreducible. Up to replacing $k$ with a field
  of definition of $f : N \r M$ which is finitely generated, we may
  assume that $k$ is finitely generated and that $\Omega$ is not only
  a universal domain but also a universal domain over $k$. Let then
  $k(X)$ be the function field of $X$ and pick an embedding $k(X)
  \subset \Omega$ which extends that of $k$.  Let us write $M=M_1
  \oplus M_2$ as in Lemma \ref{split}, with respect to the morphism $f
  : N \r M$, so that $M_1$ is isomorphic to a direct summand of $N$
  and the composite morphism $N \r M \r M_2$ is numerically
  trivial. This latter morphism induces, after base-change to
  $\Omega$, surjective maps $\CH_i(N_\Omega)_\sim \r
  \CH_i((M_2)_{\Omega})_\sim$ for all $i<l$.  Thus, up to replacing
  $M$ with $M_2$, we need only show, provided that $f$ is numerically
  trivial, that $M$ splits as $Q \oplus R(l)$ as in the theorem with
  $Q$ isomorphic to a direct summand of
  $\bigoplus_{i=n}^{l-1}\h(C_i)(i)$ for some curves $C_n,\ldots ,
  C_{l-1}$.
  
  We proceed by induction on $\dim X$.  If $l \leq n$, there is
  nothing to prove. Up to twisting, we can assume that $n=0$, and that
  $l>0$. We thus have a surjection $\CH_0(N_\Omega)_\sim \r
  \CH_0(M_\Omega)_\sim = \CH_0(X_\Omega,p_\Omega)_\sim$.  By Bertini,
  let $\iota : C \r X$ be a smooth linear section of dimension $1$ of
  $X$. By the Lefschetz hyperplane theorem, the induced map
  $\mathrm{Alb}_{C} \r \mathrm{Alb}_X$ is surjective. By functoriality
  of the albanese map, we find that $h:=f \oplus g : N \oplus \h(C) \r
  M$ induces a surjective map $(f_\Omega \oplus g_\Omega)_* :
  \CH_0(N_\Omega \oplus \h(C_{\Omega})) \r \CH_0(M_\Omega)$, where
  $g:= p \circ \Gamma_\iota : \h(C) \r M$.  Lemma \ref{lemma2} then
  implies that $(h_{k(X)})_* : \CH_0(N_{k(X)} \oplus \h(C_{k(X)})) \r
  \CH_0(M_{k(X)})$ is surjective.
  
  Let us now write $M=M' \oplus M''$ as in Lemma \ref{split}, with
  respect to the morphism $g : \h(C) \r M$, so that $M'$ is isomorphic
  to a direct summand of $\h(C)$ and the composite morphism $f \oplus
  g : N \oplus \h(C) \r M \r M''$ is numerically trivial. This latter
  morphism induces, after base-change to $k(X)$, a surjective map
  $\CH_0(N_{k(X)}\oplus \h(C_{k(X)}) ) \r \CH_0(M''_{k(X)})$.  Lemma
  \ref{prekey} implies that $\CH_0(M''_{k(X)}) = 0$.  We then deduce,
  thanks to Lemma \ref{induction}, that there exist a smooth
  projective variety $Z$ of dimension at most $\dim X - 1$ over $k$
  and an idempotent $q \in \End(\h(Z))$ such that $M'' \cong
  (Z,q,1)$. Now the motive $(Z,q)$ is such that there is a numerically
  trivial morphism $N(-1) \r (Z,q)$ inducing surjective maps
  $(f_\Omega)_* : \CH_i(N(-1)_\Omega)_{\sim} \r
  \CH_i((Z,q)_\Omega)_\sim$ for all $i < l-1$. We can thus conclude by
  the induction hypothesis that $M'' = (Z,q,1)$ splits as $P \oplus
  R(l)$, where $P$ is isomorphic to a direct summand of
  $\bigoplus_{i=1}^{l-1}\h(C_i)(i)$ for some curves $C_1,\ldots ,
  C_{l-1}$ and where $R$ is isomorphic to a direct summand of the
  motive of a smooth projective variety of dimension at most $\dim X
  -l$.  \medskip
  
  Assume now that there is an isomorphism $M \cong M^\vee(d)$. Let $M
  = Q \oplus R(l)$ be a decomposition as in the first part of the
  theorem with $R = (Z,r)$. The isomorphism $M \cong M^\vee(d)$ gives
  a morphism $N \r M \cong M^\vee(d) \r R^\vee(d) \cong
  (Z,{}^tr,d-l-\dim Z)$ which satisfies the assumptions of the first
  part of the theorem. Thus, there exist a direct summand $S$ of $N
  \oplus \bigoplus_j \h(D_j)(j)$ for some curves $D_j$, a smooth
  projective variety $Z'$ of dimension at most $ \dim Z - l +
  (d-l-\dim Z) = d-2l$ over $k$ and an idempotent $q \in
  \End(\h(Z'))$ such that the motive $R^\vee(d)$ splits as $S \oplus
  (Z',q,l)$.  Let us then define $Q' := Q \oplus S^\vee(d)$. This is a
  direct summand of $N \oplus N^\vee(d) \oplus \h(C_i)(i)$ for some
  curves $C_i$ and $M$ is isomorphic to $Q' \oplus (Z',{}^tq,l)$ with
  $\dim Z' \leq d-2l$.
  \end{proof}

  \begin{proof}[Proof of Theorems \ref{FD} and \ref{FD2}]
    If $M$ is of abelian type over $\Omega$, then a twist of $M$ is
    isomorphic to a direct summand of the motive of a product of
    curves $\prod C_i$.  Therefore, clearly, $\CH_*(M)_\sim$ is
    spanned by $\CH_*\big(\prod C_i\big)_\sim$ for any adequate
    equivalence relation. Conversely, assume that there is a
    finite-dimensional motive $N$ over $\Omega$ (e.g.\ the motive of a
    product of curves $C_i$ over $\Omega$) and a correspondence $f : N
    \r M$ such that $f_* : \CH_*(N)_\sim \r \CH_*(M)_\sim$ is
    surjective, then Theorem \ref{factor-gen} shows that $M$ is
    isomorphic to a direct summand of $N \oplus
    \bigoplus_{j}\h(D_j)(j)$ for some curves $D_j$, so that if $N$ is
    the motive of a product of curves then $M$ is of abelian type. As
    for the proof of Theorem \ref{FD2}, let $X$ be a smooth projective
    variety of dimension $d=2n$ or $2n+1$ over $\Omega$ and let $f : N
    \r \h(X)$ be a morphism inducing surjections $(f_\Omega)_* :
    \CH_i(N_\Omega)_{\sim} \r \CH_i(X_\Omega)_\sim$ for all $i \leq
    n-1$. Then, since $\h(X) \cong \h (X)^\vee(d)$, by Theorem
    \ref{factor-gen} $\h (X)$ splits as $Q \oplus R(n-1)$, where $Q$
    is a direct summand of $N \oplus N^\vee(d)$ and where $R$ is a
    direct summand of the motive of a variety of dimension at most one. In
particular, if $N$ is of abelian type, then $\h(X)$ is of
    abelian type.
  \end{proof}

  \begin{proof}[Proof of Theorem \ref{CK2}] We actually prove the
    following more general statement.  Let $X$ be a smooth projective
    variety of dimension $2n-1$ or $2n$ over $k$ and let $f : N \r
    \h(X)$ be a morphism inducing surjective maps $(f_\Omega)_* :
    \CH_i(N_\Omega)_{\sim} \r \CH_i(X_\Omega)_\sim$ for all $i \leq
    n-2$.  Assume that $N$ has a K\"unneth decomposition and is
    finite-dimensional. Then $X$ has a Chow--K\"unneth decomposition.
    Indeed, by Theorem \ref{factor-gen}, $\h(X)$ splits as a direct
    sum $Q \oplus R(n-2)$, where $Q$ is the direct summand of $N
    \oplus N^\vee(d)$, a motive that is finite-dimensional and
    satisfies the K\"unneth standard conjecture, and where $M_2$ is
    the direct summand of the motive of a curve or surface depending
    on the parity of $\dim X$. The motive $Q$ has a Chow--K\"unneth
    decomposition because it is finite-dimensional and has a K\"unneth
    decomposition (see the proof of Theorem \ref{t:fdck}). The motive $R(n-2)$
has a Chow--K\"unneth decomposition
    by \cite[Theorem 3.5]{Vial5}.
  \end{proof}

  \subsection{Applications to Chow--K\"unneth
    decompositions} \label{S:applications} As a straightforward
  consequence of Theorem \ref{factor-gen}, we deduce
  finite-dimensionality of the motive of certain $3$-folds and the
  existence of Chow--K\"unneth decomposition of certain $3$- and
  $4$-folds. Combined with results of \cite{Vial2}, we also check the
  validity of some of Murre's conjectures in those cases. The most
  general statement is the following.

\begin{theorem} \label{T:application} Let $X$ be a smooth projective
  variety over a field $k$. Assume that there exist a
  smooth projective variety $Y$ whose motive is of abelian type,
  as well as a correspondence $\Gamma \in \CH_{\dim Y}(Y \times X)$
  such that the induced map $(\Gamma_\Omega)_* : \CH_0(Y_{\Omega}) \r
  \CH_0(X_{\Omega})$ is surjective. Then, \vspace{2pt}

  \noindent $\bullet$ if $\dim X \leq 4$, $X$ has a Chow--K\"unneth
  decomposition which satisfies Murre's conjecture (B);

  \noindent $\bullet$ if $\dim X \leq 4$ and $\dim Y \leq 3$, $X$
  satisfies Murre's conjecture (D);

  \noindent $\bullet$ if $\dim X \leq 3$, $X$ is finite-dimensional;

  \noindent $\bullet$ if $\dim X \leq 3$ and $\dim Y \leq 2$, $X$
  satisfies Murre's conjecture (C).
\end{theorem}
\begin{proof}
  By Theorem \ref{factor-gen}, there is a decomposition $\h(X) = Q
  \oplus R(1)$, where $Q$ is isomorphic to a direct summand of $\h(Y)
  \oplus \h(Y)(\dim X - \dim Y)$ and where $R(1)$ is isomorphic to a
  direct summand of $\h(Z)(1)$ for some variety $Z$ of dimension at most two. By
assumption $\h(Y)$ is of abelian type, so that $Q$ is also of
  abelian type and hence has a Chow--K\"unneth decomposition. The
  motive $R(1)$ has a Chow--K\"unneth decomposition which satisfies
  Murre's conjectures (B) and (D) by \cite[Theorem 3.5]{Vial5}. Hence
  $X$ has a Chow--K\"unneth decomposition. It only remains to show
  that a motive $P$ which is of abelian type and which is isomorphic
  to the direct summand of the motive of a variety of dimension $d$,
  satisfies Murre's conjecture (B) if $d \leq 4$, satisfies Murre's
  conjecture (D) if $d \leq 3$, and satisfies Murre's conjecture (C)
  if $d \leq 2$. This is contained, respectively, in Theorem 4.5,
  Theorem 4.8 and Proposition 3.1 of \cite{Vial2}.
\end{proof}

Let $X$ be a smooth projective variety over $k$. Assume that there
exist smooth projective varieties $X_0, \ldots, X_{N-1}, X_N = X$ such
that, for all $n \leq N$, $X_n$ and $X_{n-1}$ satisfy one of the
following properties.

(1) There is a dominant rational map $\varphi_n : X_{n-1}
\dashrightarrow X_n$;

(2) There is a dominant morphism $\psi_n : X_n \rightarrow X_{n-1}$
whose generic fiber has trivial Chow group of zero-cycles after
base-change to a universal domain.

\begin{proposition} \label{P:CH0}
  Let $X$ be as above. Then there is a correspondence $\Gamma \in
  \CH^{\dim X}(X_0 \times X)$ such that $(\Gamma_\Omega)_* :
  \CH_0((X_0)_\Omega) \rightarrow \CH_0(X_\Omega)$ is surjective.
\end{proposition}
\begin{proof} Let $Y$ and $Z$ be two smooth projective varieties over
  $k$. On the one hand, a dominant rational map $\varphi : Y
  \dashrightarrow Z$ induces a surjection $\varphi_* : \CH_0(Y) \r
  \CH_0(Z)$. On the other hand, given a dominant morphism $\psi : Z \r
  Y$ as in (2) and given an ample class $h \in \CH^1(Z)$, there is a
  surjection $h^{\dim Z - \dim Y} \circ \psi^* : \CH_0(Y) \r
  \CH_0(Z)$. Here $h^{\dim Z - \dim Y}$ denotes the $(\dim Z - \dim
  Y)$-fold intersection with $h$; see \cite[Theorem 1.3]{Vial5}. This
  proves the proposition.
\end{proof}

An immediate consequence of Theorem \ref{T:application} and
Proposition \ref{P:CH0} is the following theorem.

\begin{theorem} \label{T:CKfindim}
  Let $X$ be as above and assume $X_0$ is a product of curves. Then,
  \vspace{2pt}

  \noindent $\bullet$ if $\dim X \leq 4$, $X$ has a Chow--K\"unneth
  decomposition which satisfies Murre's conjecture (B);

  \noindent $\bullet$ if $\dim X \leq 3$, $X$ is
  finite-dimensional and satisfies Murre's conjecture (D). \qed
\end{theorem}

\begin{example} Theorem \ref{T:CKfindim} notably applies to any smooth
  projective variety over $k$ which is rationally dominated by a
  product of curves. Thus a $3$-fold rationally dominated by a product
  of curves is finite-dimensional and a $4$-fold rationally
  dominated by a product of curves has a Chow--K\"unneth
  decomposition.
\end{example}

\begin{example}
  Theorem \ref{T:CKfindim} also applies to a smooth projective complex
  variety $X$ whose tangent bundle $T_X$ is nef, that is if the
  line-bundle $\mathcal{O}_{T_X}(1)$ on $\mathbb{P}(T_X)$ is nef. Indeed, a
  theorem of Demailly--Peternell--Schneider \cite{DPS} says that there
  is an \'etale cover $X' \r X$ of $X$ such that the Albanese morphism
  $X' \r \mathrm{Alb}_{X'}$ is a smooth morphism, whose fibres are
  smooth Fano varieties. Smooth Fano varieties are rationally
  connected and it follows that the generic fibre of $X' \r
  \mathrm{Alb}_{X'}$ is rationally connected and hence has trivial
  Chow group after base-change to a universal domain. Note that the
  case of a $3$-fold with a nef tangent bundle was taken care of, by a
  different method relying on a stronger classification result due to
  Campana--Peternell, in \cite{Iyer}.
\end{example}

\subsection{Application to smash-nilpotent $1$-cycles} \label{S:smash}
R. Sebastian \cite{sebastian} proved that every $1$-dimensional cycle
on a product of curves which is numerically trivial is
smash-nilpotent. Theorem \ref{factor-gen} makes it possible to extend
Sebastian's result.

\begin{theorem} \label{T:smash} Let $C_i$ be smooth projective curves
  and let $f: \h (\prod_i C_i) \r M$ be a morphism of effective
  motives over $k$. Assume that $(f_\Omega)_* : \CH_0(\prod_i
  C_{i,\Omega}) \r \CH_0(M_\Omega)$ is surjective. Then
  smash-nilpotence equivalence agrees with numerical equivalence on
  $1$-cycles on $M$.
\end{theorem}
\begin{proof}
  By Theorem \ref{factor-gen}, there is a splitting $M \cong Q \oplus
  R(1)$, where $Q$ is a direct summand of $ \h (\prod_i C_i)$ and
  where $R$ is effective.
  Therefore, $\CH_1(M)_{\mathrm{num}}$ is spanned via the action of
  correspondences by $\CH_1(\prod_i C_i)_{\mathrm{num}} \oplus
  \CH_0(R)_{\mathrm{num}}$. Numerically trivial cycles in $\CH_0(R)$
  are clearly smash-nilpotent and, by Sebastian's theorem, numerically
  trivial cycles in $\CH_1(\prod_i C_i)$ are smash-nilpotent.  The
  theorem is thus proved.
\end{proof}

\subsection{A splitting result without
  finite-dimensionality} \label{S:revisited} Let $f : N \r M$ be a
morphism of motives over $k$ such that $(f_\Omega)_* : \CH_i(N_\Omega)
\r \CH_i(M_\Omega)$ is surjective for all $i$ and assume that $N$ is
finite-dimensional. An immediate corollary to Theorem
\ref{factor-gen} (or rather its proof) is that $f$ has a
right-inverse.  In this paragraph, we show that, in the situation
above, the assumption that $N$ be finite-dimensional can be dropped.

\begin{theorem} \label{fdsummand2} Let $f : N \r M$ be a morphism of
  motives over $k$ such that $(f_\Omega)_* : \CH_*(N_\Omega) \r
  \CH_*(M_\Omega)$ is surjective. Then $f$ has a right-inverse.
\end{theorem}
\begin{proof}
   Up to replacing $k$ with a field of definition of $f$ which is
  finitely generated, we may assume that $\Omega$ is a universal
  domain over $k$.  Let's write $M=(X,p,n)$ and $N=(Y,q,m)$. Let $Z$
  be an irreducible variety over $k$ and let's define, as in
  \cite[proof of Lemma 1]{GG}, an action $f \otimes Z : \CH_*(Y \times
  Z) \r \CH_*(X \times Z)$ by $(f \otimes Z) \alpha :=
  (p_{X,Z})_*(p_{Y,Z}^* \alpha \cdot_{p_{Y,X}} f)$. Here, $p_{X,Y}$,
  $p_{Y,Z}$ and $p_{X,Z}$ are the natural projections, and the product
  ``$\cdot_{p_{Y,X}}$'' is Fulton's refined intersection \cite[\S
  8]{Fulton} with respect to the projection $p_{Y,X} : Y \times X
  \times Z \r Y \times X$. By \cite[Corollary 8.1.2]{Fulton}, it can
  be checked that, when $Z$ is smooth (projective), we have $(f
  \otimes Z) \alpha = (f,\id_Z)_*\alpha$.
  % \footnote{proof???}

  We are going to prove by induction on $\dim Z$ that $f \otimes Z :
  \CH_*(Y \times Z) \r \CH_*(X \times Z)$ has same image as $p \otimes
  Z : \CH_*(X \times Z) \r \CH_*(X \times Z)$. Let $K$ be the function
  field of $Z$ over $k$. We have a diagram with exact rows
  \begin{center}
    $\xymatrix{ \bigoplus \CH_*(Y \times D) \ar[d]^{\bigoplus f
        \otimes D} \ar[r] & \CH_*(Y \times Z) \ar[d]^{f \otimes Z}
      \ar[r] & \CH_*(Y_K) \ar[d]^{f\otimes K} \ar[r] & 0 \\
      \bigoplus \CH_*(X \times D) \ar[r] & \CH_*(X \times Z) \ar[r] &
      \CH_*(X_K) \ar[r] & 0. }$
  \end{center}
  The direct sums are taken over all closed irreducible
  codimension-$1$ subschemes of $Z$ and the left horizontal arrows are
  induced by the natural inclusions. Moreover, the diagram is
  commutative. The left-hand square is commutative by the projection
  formula \cite[Proposition 8.1.1.(c)]{Fulton} and the right-hand
  square is commutative by \cite[Theorem 6.4]{Fulton} and by
  \cite[Proposition 8.1.1.(d)]{Fulton} which is valid for ``regular
  imbedding'' replaced by ``l.c.i morphism'' thanks to
  \cite[Proposition 6.6.(c)]{Fulton}.
  By assumption and by Lemma \ref{lemma2},
  for all $\alpha \in \CH_*(X_K)$, there is a $\beta \in \CH_*(Y_K)$
  such that $(p \otimes K)\alpha = (f \otimes K) \beta$.  By
  induction, there is also, for all irreducible divisors $D$ in $Z$
  and all $\gamma \in \CH_*(X \times D)$, a cycle $\delta \in \CH_*(Y
  \times D)$ such that $(p\otimes D)\gamma = ( f\otimes D) \delta$.
  Therefore, by a simple diagram-chase and by commutativity of the
  similar diagram involving $\bigoplus p \otimes D$, $p \otimes Z$ and
  $p \otimes K$, we see that the middle vertical map has image
  coinciding with the image of $p \otimes Z$.

  Thus, $(f \times \id_Z)_* : \CH_*(N \otimes \h(Z)) \r \CH_*(M
  \otimes \h(Z))$ is surjective for all smooth projective varieties
  $Z$ over $k$. By Kimura \cite[Lemma 6.8]{Kimura}, it follows that
  $f$ has a right-inverse.
\end{proof}

\subsection{Homological motives and splittings} \label{S:hom} Theorem
\ref{fdsummand2} can be thought of as an analogue modulo rational
equivalence of Theorem \ref{T:splithomol} below which is concerned
with motives modulo homological equivalence. The arguments used in
this section go back to Andr\'e \cite{Andre96} and improve slightly on
the results of Arapura \cite{Arapura}. The difference with
\cite{Arapura} is that we are able to ignore the middle cohomology of
$X$; see Proposition \ref{Ar2}.

Recall \cite{Kleiman} that, in characteristic zero, for a smooth
projective variety, the standard conjectures reduce to the Lefschetz
standard conjecture.  Given a smooth projective variety $X$ of
dimension $d$ over $k$, we write $H^i(X)$ for the Betti cohomology of
$X$ with rational coefficients and we write $H_i(X)$ for
$H^{2d-i}(X)$.

\begin{proposition} \label{Ar2} Assume that $\mathrm{char} \, k = 0$.
  Let $X$ and $Y$ be smooth projective varieties over $k$ such that
  there exists a morphism $f : \bigoplus_{m,n} \h(Y)^{\otimes n}(m) \r
  \h(X)$ that induces a surjection on $H_i(X)$ for all $i > \dim X$.
  Assume that $Y$ satisfies the standard conjectures. Then $X$ also
  satisfies the standard conjectures.
\end{proposition}
\begin{proof} If $Y$ satisfies the standard conjectures, then so do
  all of its powers. For varieties satisfying the standard
  conjectures, the Lefschetz involution $*_L$ on cohomology is induced
  by a correspondence; see \cite{Andre96} and \cite{Kleiman}. Let's
  fix a polarization on $Y$ and let's endow $Y^n$ with the product
  polarization. Let $f_{m,n}$ be the restriction of $f$ to $
  \h(Y)^{\otimes n}(m)$ and, for $i > \dim X$, consider the composite
  map $$\varphi_{m,n} : H^i(X) \arr{f_{m,n}^*} H^{i-2m}(Y^n) \arr{L}
  H_{i-2m}(Y^n) \arr{*_L} H_{i-2m}(Y^n) \arr{(f_{m,n})_*} H_i(X) .$$
  Here the map $L$ is the Lefschetz isomorphism with respect to the
  polarization on $Y^n$. It is induced by a correspondence by
  assumption. Since $H_{i-2m}(Y^n)$ is polarized with respect to the
  bilinear form $\langle\, L^{-1}- ,*_L- \, \rangle$, we obtain that
  $\varphi_{m,n}$ has same image as $(f_{m,n})_*$. Thus, there exists
  $h : \h(X)(-d) \r \bigoplus_{m,n} \h(Y)^{\otimes n}(m)$ such that $f
  \circ h$ induces a surjective map $H^i(X) \r H_i(X)$. It is then
  classical to deduce, thanks to the theorem of Cayley--Hamilton, that
  the inverse to the Lefschetz isomorphism $H_i(X) \r H^i(X)$ is
  induced by a correspondence. Thus, $X$ satisfies the Lefschetz
  standard conjecture and hence the standard conjectures.
\end{proof}

\begin{remark} \label{lef} In some cases, Proposition \ref{Ar2} can be
  slightly improved. Indeed, if $Y$ is a $3$-fold, although $Y$ might
  not be known to satisfy the Lefschetz standard conjecture, it is
  known that the Lefschetz involution $*_L$ on $H^3(Y)$ is algebraic.
  It is thus possible to prove the following statement.  Let $X$ be a
  smooth projective variety of dimension $d$ defined over a subfield
  of $\C$ and let $i>d$.  Assume that $H_i(X) = \widetilde{N}^{\lfloor
    i/2 \rfloor -1} H_i(X)$. Here $\N$ denotes the niveau filtration
  defined in \cite{Vial2}.  In other words assume that \vspace{2pt}

  $\bullet$ if $i$ is odd, then there exist a threefold $Y_i$ and a
  correspondence $\Gamma_i \in \CH_{(i+3)/2}(Y_i \times X)$ such
  that $(\Gamma_i)_* : H_3(Y_i) \r H_i(X)$ is surjective, and

  $\bullet$ if $i$ is even, then there exist a surface $Z_i$ and a
  correspondence $\Gamma_i \in \CH_{(i+2)/2}(Z_i \times X)$ such
  that $(\Gamma_i)_* : H_2(Z_i) \r H_i(X)$ is surjective.
  \vspace{2pt}

  \noindent Then the inverse to the Lefschetz isomorphism $L^{d-i} :
  H_i(X) \r H^i(X)$ is induced by an algebraic correspondence.\qed
  \end{remark}

A combination of Proposition \ref{BS} and Proposition \ref{Ar2} gives a
criterion on $\CH_0(X)$ for a complex fourfold $X$ to satisfy the
Lefschetz standard conjecture.

\begin{proposition} \label{stdth} Let $X$ be a smooth projective
  variety of dimension $d \leq 4$ over $\C$. Assume that there exist a
  smooth projective variety $Y$ which satisfies Grothendieck's
  Lefschetz standard conjecture, as well as a correspondence $\Gamma
  \in \CH_{\dim Y}(Y \times X)$ such that $\Gamma_* : \CH_0(Y) \r
  \CH_0(X)$ is surjective.  Then $X$ satisfies the Lefschetz standard
  conjecture.
 \end{proposition}

\begin{proof} By Lemma \ref{lemma} and Proposition \ref{BS}, we get that
$$\Delta_X = \Gamma\circ \Gamma_1 + \Gamma_2 \in \CH_d(X \times X),$$
for some correspondence $\Gamma_1 \in \CH_{d}(X \times Y)$ and some
correspondence $\Gamma_2 \in \CH_{d}(X \times X)$ that
factors as $\Gamma_2 = \alpha \circ \beta$ for some smooth projective
variety $Z$ and for some $\beta \in \CH^d(X \times Z)$ and some
$\alpha \in \CH_d(Z \times X)$.  By looking at the action of
$\Delta_X$ on $H^i(X)$, we see that
$$H^i(X) = \Gamma_1^*H^i(Y) + \beta^* H^{i-2}(Z).$$ This
settles the proposition thanks to Proposition \ref{Ar2}.
\end{proof}

\begin{theorem} \label{T:splithomol} Assume that $\mathrm{char} \, k =
  0$.  Let $X$ and $Y$ be smooth projective varieties over $k$ such
  that there exists a morphism $f : \bigoplus_{m,n} \h(Y)^{\otimes
    n}(m) \r \h(X)$ that induces a surjection on $H_i(X)$ for all $i$.
  Assume that $Y$ satisfies the standard conjectures. Then $f$ has a
  right-inverse modulo homological equivalence. If, moreover, $\h(X)$
  is finite-dimensional, then $f$ has a right-inverse.
\end{theorem}
\begin{proof}
  By Proposition \ref{Ar2}, $X$ satisfies the standard conjectures.
  By Andr\'e \cite[\S 4]{Andre96}, the full, thick and rigid
  sub-category of homological motives spanned by the motives of $X$
  and $Y$ is semi-simple. Therefore, $f$ has a right-inverse modulo
  homological equivalence. The last statement follows from Lemma
  \ref{split}.
\end{proof}

\begin{proof}[Proof of Theorem \ref{T:abhomo}] This is a direct
  consequence of Theorem \ref{T:splithomol} because a motive of
  abelian type is a direct factor of the motive of a product of curves
  and because a product of curves satisfies the standard conjectures.
\end{proof}

\vspace{7pt}
\def\cprime{$'$}

\end{document}